\makeatletter
\@namedef{ver@picins.sty}{9999/99/99}
\makeatother
\documentclass{SCIYA2017enOL}
\usepackage{amsmath,amssymb,amsthm,amscd}
\usepackage{color}
\usepackage[colorlinks=true]{hyperref}

\online

\begin{document}

\ensubject{fdsfd}

\ArticleType{ARTICLES}
\Year{2023}
\Month{}%
\Vol{}
\No{}
\BeginPage{1} %
\DOI{}
\ReceiveDate{}
\AcceptDate{}
\OnlineDate{}

\title[]{Multiplicity of solutions for semilinear subelliptic Dirichlet
problem}{Multiplicity of solutions for semilinear subelliptic Dirichlet
problem}

\author[1,$\ast$]{Hua Chen}{{chenhua@whu.edu.cn}}
\author[2]{Hong-Ge Chen}{{hongge\_chen@whu.edu.cn}}
\author[1]{Jin-Ning Li}{lijinning@whu.edu.cn}
\author[1]{Xin Liao}{xin\_liao@whu.edu.cn}

\AuthorMark{Chen H}

\AuthorCitation{Chen H, Chen H G, Li J N, Liao X}

\address[1]{School of Mathematics and Statistics, Wuhan University, Wuhan {\rm430072}, China}
\address[2]{Wuhan Institute of Physics and Mathematics,\\
 Innovation Academy for Precision Measurement Science and Technology,\\
  Chinese Academy of Sciences, Wuhan {\rm430071}, China}

\abstract{In this paper, we study the semilinear
subelliptic equation
\[ \left\{
      \begin{array}{cc}
      -\triangle_{X} u=f(x,u)+g(x,u) & \mbox{in}~\Omega, \\[2mm]
      u=0\hfill           & \mbox{on}~\partial\Omega,
      \end{array}
 \right.  \]
where $\triangle_{X}=-\sum_{i=1}^{m}X_{i}^{*}X_{i}$ is the self-adjoint H\"{o}rmander operator associated with vector fields $X=(X_{1},X_{2},\ldots,X_{m})$ satisfying the H\"{o}rmander condition, $f(x,u)\in C(\overline{\Omega}\times \mathbb{R})$, $g(x,u)$ is a Carath\'{e}odory function on $\Omega\times \mathbb{R}$, and $\Omega$ is an open bounded domain in $\mathbb{R}^n$ with smooth boundary. Combining the perturbation from symmetry method with the approaches involving eigenvalue estimate and Morse index in estimating the min-max values, we obtain two kinds of existence results for multiple weak solutions to the problem above. Furthermore, we discuss the difference between the eigenvalue estimate approach and the Morse index approach in degenerate situations. Compared with the classical elliptic cases, both approaches here have their own strengths in the degenerate cases. This new phenomenon implies the results in general degenerate cases would be quite different from the situations in classical elliptic cases.}



\maketitle

\section{Introduction and main results}
For $n\geq 2$, let $X=(X_{1},X_{2},\ldots,X_{m})$ be a system of real smooth vector fields defined on an open domain $W\subset \mathbb{R}^{n}$. Suppose that $X=(X_{1},X_{2},\ldots,X_{m})$ satisfy the so-called H\"ormander's condition (cf. \cite{hormander1967}) in $W$, i.e., the vector fields $X_{1}, X_{2},\ldots, X_{m}$ together with their commutators up to a certain fixed length span the tangent space at each point of $W$. Then, we denote by $ \triangle_{X}:=-\sum_{i=1}^{m}X_{i}^{*}X_{i}$ the H\"ormander operator associated with vector fields $X=(X_{1},X_{2},\ldots,X_{m})$, where $X_{i}^{*}$ is the formal adjoint of $X_{i}$.

As an important class of degenerate elliptic operators, the H\"{o}rmander operators have received considerable attention
over the last fifty years. The motivations in studying H\"{o}rmander operators originated from its wide applications among many different areas not only in PDEs but also in sub-Riemannian geometry, systems of stochastic differential equations, the theory of functions of several complex variables, geometric control theory and nonholonomic mechanics, see \cite{Montgomery2002,Bell1996,Jean2014,Hormander1973}. Through the work of several mathematicians, many remarkable results have been acquired, such as subelliptic estimates, Poincar\'{e} inequalities, Sobolev embedding results, Harnack's inequalities, size estimates of the Green kernels and heat kernels, etc. Here we mention the works \cite{hormander1967,Stein1976,Jerison1986duke,Yung2015,Citti1993,Sanchez-calle1984,Stein1985, Jerison1986,Jerison1987,MB2014,Capogna1993,Capogna1994,Capogna1996,Cohn2007,Derridj1971,Derridj1972,Garofalo1996,Jost1998} and the references therein.

In this paper, we study the semilinear subelliptic Dirichlet problem
\begin{equation}\label{problem1-1}
\left\{
      \begin{array}{cc}
      -\triangle_{X} u=f(x,u)+g(x,u) & \mbox{in}~\Omega, \\[2mm]
      u=0\hfill           & \mbox{on}~\partial\Omega.
      \end{array}
 \right.
\end{equation}
Here $\Omega\subset\subset W$ is a bounded
  open domain with smooth boundary, and $f\in  C(\overline{\Omega}\times\mathbb{R})$ is a continuous function satisfying the following assumptions:
   \begin{enumerate}

            \item [(H.1)] $f$ is odd with respect to variable $u$, i.e.~ $f(x,-u)=-f(x,u)$.

          \item [(H.2)] There exist $2<p<2_{\tilde{\nu}}^{*}:=\frac{2\tilde{\nu}}{\tilde{\nu}-2}$ and $C>0$ such that
         \[ |f(x,u)|\leq C(1+|u|^{p-1}) \]
   holds for all $x\in \overline{\Omega}$ and all $u\in \mathbb{R}$, where $\tilde{\nu}\geq 3 $ is the generalized M\'{e}tivier index (also called the non-isotropic dimension of $\Omega$ related to the vector fields $X$) defined in Definition \ref{def2-1}  below.

   \item [(H.3)] There exist $\mu>2$ and $R_0>0$ such that
\[0<\mu F(x,u)\leq f(x,u)u\]
   holds for all  $x\in\overline{\Omega}$ and all $u\in \mathbb{R}$ with $|u|\geq R_0$, where $F(x,u)=\int_{0}^{u}f(x,v)dv$ is the primitive of $f(x,u)$.
\end{enumerate}
  Moreover, we suppose that $g:\Omega\times \mathbb{R}\to \mathbb{R}$ is a Carath\'{e}odory function, i.e., $g(x,\cdot)$ is a continuous function  for almost everywhere $x\in \Omega$, and $g(\cdot,u)$ is a measurable function for all $u\in \mathbb{R}$. Additionally,
   \begin{enumerate}
            \item [(H.4)] There exist $0\leq\sigma<\mu-1$ and  $\beta\geq 0$ such that
         \[ |g(x,u)|\leq \alpha(x)+\beta|u|^{\sigma} \]
       holds  for all $x\in \Omega, ~u\in \mathbb{R}$, and some non-negative $\alpha(x) \in L^{\frac{\mu}{\mu-1}}(\Omega)$.
\end{enumerate}

If $X=(\partial_{x_{1}},\ldots,\partial_{x_{n}})$, $\triangle_{X}$ coincides with the classical Laplacian $\triangle$ and the generalized M\'{e}tivier index $\tilde{\nu}=n$. The classical semilinear elliptic Dirichlet problems have been intensively studied over the past half-century. In 1981,  Bahri-Berestycki  \cite{Bahri-Berestycki1981} first investigated the problem \eqref{problem1-1} with $f(x,u)=|u|^{p-2}u$ and $g(x,u)=g(x)\in L^2(\Omega)$, and they proved the problem \eqref{problem1-1} has infinitely many weak solutions provided
\begin{equation}\label{1-1-12}
  2<p<\frac{5n-2+\sqrt{9n^{2}-4n+4}}{4(n-1)}<\frac{2n}{n-2}.
\end{equation}
In addition, if $2<p<\frac{2n}{n-2}$, Bahri \cite{Bahri1981} showed that
the problem \eqref{problem1-1} has infinitely many weak solutions for a residual set of $g$ in $H^{-1}(\Omega)$. Meanwhile, similar multiplicity results were obtained by Struwe \cite{Struwe1980} and Dong-Li \cite{Dong1982} via the approach in the spirit of symmetric mountain pass theorem. Then, by the perturbation from symmetry method, Rabinowitz \cite{Rabinowitz1982,Rabinowitz1986} studied the problem \eqref{problem1-1} for general function $f$ under the assumptions (H.1)-(H.3) and $g(x,u)=g(x)\in L^2(\Omega)$. More precisely, he established the multiplicity of weak solutions under
\begin{equation}\label{1-1-13}
  2<p<\frac{\mu n+(\mu-1)(n+2)}{\mu n+(\mu-1)(n-2)}+1<\frac{2n}{n-2},
\end{equation}
where the admissible range of index $p$ in \eqref{1-1-13} was carried out by carefully estimates on min-max values obtained through the lower bound estimate $\lambda_k\geq Ck^{\frac{2}{n}}$ on $k$-th Dirichlet eigenvalue $\lambda_{k}$ of Laplacian. Furthermore, Bahri-Lions \cite{Bahri1988} improved these multiplicity results and dealt with the general perturbation terms by combining the perturbation method with the sharp estimates on min-max values obtained through the Cwikel-Lieb-Rozenblum inequality and an estimate for the Morse index of the corresponding critical points. For example, if $f(x,u)=|u|^{p-2}u$ and $g(x,u)=g(x)\in L^2(\Omega)$, \cite{Bahri1988} gave the multiplicity of weak solutions for
\begin{equation}\label{1-1-14}
  2<p<\frac{2n-2}{n-2}<\frac{2n}{n-2}.
\end{equation}
Observing that $\mu=p$ when $f(x,u)=|u|^{p-2}u$, the admissible range of $p$ given by \eqref{1-1-13} is the same with \eqref{1-1-12}, and smaller than \eqref{1-1-14} since  $\frac{5n-2+\sqrt{9n^{2}-4n+4}}{4(n-1)}<\frac{2n-2}{n-2}$ for $n\geq 3$. It turns out that the Morse index approach on estimates of min-max values usually yields better results in studying the multiple weak solutions of problem \eqref{problem1-1} for classical Laplacian. For more multiplicity results related to the classical semilinear elliptic equations with perturbation terms, one can see \cite{Struwe2000,Sqassina2006,Bahri1992,Tanaka1989} and the references therein.

In recent years, the study of nonlinear degenerate elliptic equations on sub-Riemannian manifolds has become an active area in PDEs and geometry analysis. In the degenerate case, $\triangle_{X}$ is a subelliptic operator and the generalized M\'{e}tivier index $\tilde{\nu}>n\geq 2$. For the vector fields $X$ under the M\'etivier's condition (see \cite{Metivier1976} and Remark \ref{remark2-1} below), the existence and multiplicity of weak solutions for the nonlinear degenerate elliptic equations involving H\"{o}rmander vector fields have been treated by Jerison-Lee \cite{Jerison-lee1987}, Garofalo-Lanconelli \cite{Garofalo1992}, Xu \cite{Xu1994,Xu1995,Xu1996}, Xu-Zuily \cite{Xu1997}, Citti \cite{Citti1995}, Maalaoui-Martino \cite{Maalaoui2013}, Loiudice \cite{Loiudice2007} and Venkatesha Murthy \cite{Murthy2008}.
 However, when we deal with the general cases without M\'etivier's condition, the lack of degenerate Rellich-Kondrachov compact embedding results and precise lower bound of Dirichlet eigenvalues caused many obstacles
over a long period in the past. As far as we know, there are only few results in this situation. The multiplicity of weak solutions for
problem \eqref{problem1-1} with perturbation term have been studied by Luyen-Tri \cite{Luyen2019} for some special Grushin type operators. In addition, Chen-Chen-Yuan \cite{Chen-Chen-Yuan2022} generalized Rabinowitz's classical results in \cite{Rabinowitz1982,Rabinowitz1986} to general degenerate subelliptic operators and gave the multiplicity results of problem \eqref{problem1-1} with the free perturbation term $g(x,u)=g(x)\in L^{2}(\Omega)$. Specifically, if the boundary $\partial\Omega$ is non-characteristic for $X$ (i.e. for any $x_{0}\in\partial\Omega$, there exists at least one vector field $X_{j_{0}}~(1\leq j_0\leq m)$ such that $X_{j_{0}}(x_{0})\notin T_{x_{0}}(\partial\Omega)$), \cite{Chen-Chen-Yuan2022} proved the problem \eqref{problem1-1} possesses a sequence of weak solutions if nonlinear term $f$ satisfies assumptions (H.1)-(H.3) and
\begin{equation}\label{1-1-5}
  2<p< \frac{\mu \tilde{\nu}+(\mu-1)(\tilde{\nu}+2)}{\mu \tilde{\nu}+(\mu-1)(\tilde{\nu}-2)}+1<\frac{2\tilde{\nu}}{\tilde{\nu}-2}.
\end{equation}

It is worth pointing out that the admissible range of index $p$ in \eqref{1-1-5} depends on the explicit lower bound $\lambda_k\geq Ck^{\frac{2}{\tilde{\nu}}}$ for $k$-th Dirichlet eigenvalue $\lambda_{k}$ of H\"{o}rmander operator $\triangle_{X}$ obtained in \cite{Chen2021} in which the non-characteristic condition for $\partial\Omega$ would be also assumed. However, in the absence of M\'{e}tivier condition, the Weyl type asymptotic formula of Dirichlet eigenvalues for general H\"{o}rmander operator has not yet been fully understood, and such lower bound of $\lambda_k$ may be imprecise in some degenerate cases. For example, the recent study \cite{Chen-Chen-Li2022} proved that for a class of homogeneous H\"{o}rmander operators on general smooth domain  without the non-characteristic condition, the Dirichlet eigenvalue $\lambda_{k}$ admits the asymptotic behaviour  $\lambda_k\approx k^{\frac{2}{\vartheta}}(\ln k)^{-\frac{2d}{\vartheta}}$ as $k\to+\infty$, where $n\leq \vartheta\leq \tilde{\nu}$ is a positive rational number and $0\leq d\leq n-1$ is an integer. In view of this asymptotic behaviour, it is extremely interesting to analyse the effects of the lower bound of $\lambda_{k}$ on the admissible range of index $p$. One can find some new situations that never happen in non-degenerate cases.

The present work is devoted to investigating the admissible range of index $p$ such that the semilinear subelliptic Dirichlet problem \eqref{problem1-1}  preserves a sequence of weak solutions. First, under a general setting on the lower bound of Dirichlet eigenvalues, we have

\begin{theorem}\label{thm1}
Let $X=(X_{1},X_{2},\ldots,X_{m})$ be the  real smooth vector fields defined on an open domain
  $W\subset\mathbb{R}^n$ and satisfy the H\"{o}rmander's condition in $W$.  Assume that $\Omega\subset\subset W$ is a bounded open domain with smooth boundary. Denoting by $\lambda_k$  the $k$-th Dirichlet eigenvalue of $-\triangle_{X}$ on $\Omega$, we suppose $\lambda_k$ admits the following lower bound:
\[ \lambda_k\geq Ck^{\frac{2}{\vartheta}}(\ln k)^{-\kappa}~~~\mbox{for sufficient large}~~k,\leqno(L)\]
 where $n\leq\vartheta\leq \tilde{\nu}$, $\kappa\geq 0$ and $C>0$ are the constants depending on $X$ and $\Omega$. If the functions $f$ and $g$ in the problem \eqref{problem1-1} satisfy the assumptions $(H.1)$-$(H.4)$ and
\[ \frac{2p}{\vartheta(p-2)}-\frac{\tilde{\nu}}{\vartheta}>\frac{\mu}{\mu-\sigma-1},\leqno(A1)\]
then the problem \eqref{problem1-1} possesses an unbounded sequence of weak solutions in $H_{X,0}^{1}(\Omega)$, where $H_{X,0}^{1}(\Omega)$ is the function space defined in Section \ref{Section2} below.
\end{theorem}

\begin{remark}
It is worth emphasizing that, compared with the result in \cite{Chen-Chen-Yuan2022},  the main result of Theorem \ref{thm1} in this paper would be more interesting and reasonable since we study problem \eqref{problem1-1} without the non-characteristic condition for boundary $\partial\Omega$. In this case, we need to prove Friedrichs-Poincar\'{e} type inequality related to H\"ormander's vector fields for general bounded domain $\Omega$ with smooth boundary (see Proposition \ref{Poincare} below) and use a sharp lower bound of Dirichlet eigenvalues in the form $\lambda_{k}\geq C k^{\frac{2}{\vartheta}}(\ln k)^{-\frac{2d}{\vartheta}}$ for a wide class of homogeneous H\"{o}rmander operators in the recent result of  \cite{Chen-Chen-Li2022}. We remark that the homogeneous H\"{o}rmander operators include a lot of well-known degenerate operators, such as Grushin operators, Bony operators, Martinet operators, etc (see examples in \cite{Chen-Chen-Li2022}). Note that the inequality $(A1)$ in Theorem \ref{thm1} is equivalent to
\begin{equation}\label{1-1-16}
2<p<\frac{\vartheta\mu+(\tilde{\nu}+2)(\mu-\sigma-1)}{\vartheta\mu+(\tilde{\nu}-2)(\mu-\sigma-1)}+1<\frac{2\tilde{\nu}}{\tilde{\nu}-2}.
\end{equation}
Since $n\leq \vartheta\leq \tilde{\nu}$, we have
\[ \frac{\tilde{\nu}\mu+(\tilde{\nu}+2)(\mu-\sigma-1)}{\tilde{\nu}\mu+(\tilde{\nu}-2)(\mu-\sigma-1)}\leq \frac{\vartheta\mu+(\tilde{\nu}+2)(\mu-\sigma-1)}{\vartheta\mu+(\tilde{\nu}-2)(\mu-\sigma-1)}.\]
Therefore, Theorem \ref{thm1} may derive a wider admissible range of index $p$ than the results in  \cite{Chen-Chen-Yuan2022, Luyen2019}.
\end{remark}

Secondly, according to the Morse index approach and degenerate Cwikel-Lieb-Rozenblum inequality, we can obtain the following result without the condition $(L)$.
\begin{theorem}\label{thm2}
Let $X=(X_{1},X_{2},\ldots,X_{m})$ be the real smooth vector fields defined on an open domain
  $W\subset\mathbb{R}^n$, which satisfy the H\"{o}rmander's condition in $W$.  Suppose that  $\Omega\subset\subset W$ is a bounded open domain with smooth boundary. Assume that the functions $f$ and $g$ in the problem \eqref{problem1-1} satisfy the assumptions (H.1)-(H.4) and
\[ \frac{2p}{\tilde{\nu}(p-2)}>\frac{\mu}{\mu-\sigma-1},\leqno(A2)\]
then the problem \eqref{problem1-1} possesses an unbounded sequence of weak solutions in $H^{1}_{X,0}(\Omega)$.
\end{theorem}

\begin{remark}
Inequality $(A2)$ is equivalent to
\begin{equation}\label{1-1-17}
2<p<\frac{\mu\tilde{\nu}+2(\mu-\sigma-1)}{\mu\tilde{\nu}-2(\mu-\sigma-1)}+1<\frac{2\tilde{\nu}}{\tilde{\nu}-2}.
\end{equation}
For $n\leq \vartheta\leq \tilde{\nu}$, comparing \eqref{1-1-16} with \eqref{1-1-17}, we have
\begin{equation}\label{1-1-18}
\begin{aligned}
&\frac{\vartheta\mu+(\tilde{\nu}+2)(\mu-\sigma-1)}{\vartheta\mu+(\tilde{\nu}-2)(\mu-\sigma-1)}-\frac{\mu\tilde{\nu}+2(\mu-\sigma-1)}{\mu\tilde{\nu}-2(\mu-\sigma-1)}\\
&=\frac{4(\mu-\sigma-1)((\sigma+1)\tilde{\nu}-\vartheta\mu)}{(\vartheta\mu+(\tilde{\nu}-2)(\mu-\sigma-1))(\mu\tilde{\nu}-2(\mu-\sigma-1))}.
\end{aligned}
\end{equation}
If $(\sigma+1)\tilde{\nu}<\vartheta \mu$ (e.g. in case of $\vartheta=\tilde{\nu}$), \eqref{1-1-18} implies the admissible range of index $p$ given by Theorem \ref{thm2} is wider than that in Theorem \ref{thm1} since the  supremum of $p$ in \eqref{1-1-17} is closer to the critical index $\frac{2\tilde{\nu}}{\tilde{\nu}-2}$. In this case,
 Theorem \ref{thm2} improves  Theorem \ref{thm1} (also improves the results in \cite{Chen-Chen-Yuan2022,Luyen2019}), which is analogous to the classical elliptic case (cf. \cite{Bahri1988,Rabinowitz1982,Rabinowitz1986}).
However, unlike the classical elliptic cases, there exists some degenerate examples (e.g. Example \ref{ex5-1} below) satisfying $(\sigma+1)\tilde{\nu}>\vartheta \mu$, which indicates the admissible range of index $p$ given in \eqref{1-1-16} will be wider than that in \eqref{1-1-17}  sometimes, and Theorem \ref{thm2} cannot cover Theorem \ref{thm1}. From this point of view, we know that both approaches in Theorem \ref{thm1} and Theorem \ref{thm2}, respectively, have their own strengths in the degenerate cases. Also, such a new phenomenon shows the multiplicity results of problem \eqref{problem1-1} in the degenerate cases would be quite different from the situations in classical elliptic cases.
\end{remark}

The rest of our paper is organized as follows. In Section \ref{Section2}, we introduce some preliminaries including the degenerate Rellich-Kondrachov compact embedding theorem, the degenerate Cwikel-Lieb-Rozenblum inequality, and some prior estimates in critical point theory. Then in Section \ref{Section3}, we prove Theorem \ref{thm1} by combining the perturbation from symmetry method with the eigenvalue estimate approach. Next, we give the proof of Theorem \ref{thm2} in Section \ref{Section4} by invoking the Morse index theory. Finally, we present an example in Section \ref{Section5} to flourish the applications of Theorem \ref{thm1} and Theorem \ref{thm2}.

\textbf{Notations:} For the sake of simplicity, different positive constants are usually denoted by $C$ sometimes without indices.

\section{Preliminaries}
\label{Section2}
 In this section, we recall some useful results of H\"{o}rmander vector fields and present some prior estimates in critical point theory.

\subsection{Some results of H\"{o}rmander vector fields}

If the vector fields $X=(X_{1},X_{2},\ldots, X_{m})$ satisfy the H\"{o}rmander's condition on an open domain $W$ in $\mathbb{R}^n$, then for any bounded domain $\Omega\subset\subset W$, there exists a smallest positive integer $Q\geq 1$  such that $X_{1},X_{2},\ldots,X_{m}$ together with their commutators of length at most $Q$ span the tangent space $T_{x}(W)$ at each point  $x\in \overline{\Omega}$. The integer $Q$ is called the  H\"{o}rmander's index of $\overline{\Omega}$ with respect to $X$. For the vector fields $X$ under H\"{o}rmander's condition, we can define the generalized M\'{e}tivier index, which is also known as the non-isotropic dimension of $\Omega$ related to $X$ (cf. \cite{Chen2021,Yung2015}).

 \begin{definition}[Generalized M\'{e}tivier index]
    \label{def2-1}
    For each $x\in \overline{\Omega}$ and $1\leq j\leq Q$, let $V_{j}(x)$
be the subspace of the tangent space at $x$  spanned by all
commutators of $X_{1},\ldots,X_{m}$ with length at most $j$. We denote by $\nu_{j}(x)$ the dimension of vector space $V_{j}(x)$ at  $x\in \overline{\Omega}$. The pointwise homogeneous dimension at $x$ is given
by
\begin{equation}\label{2-1}
  \nu(x):=\sum_{j=1}^{Q}j(\nu_{j}(x)-\nu_{j-1}(x)),\qquad \nu_{0}(x):=0.
\end{equation}
Then we define
\begin{equation}\label{2-2}
  \tilde{\nu}:=\max_{x\in\overline{\Omega}} \nu(x)
\end{equation}
as the generalized M\'{e}tivier index of $\Omega$ associated with the vector fields $X$. The index $\tilde{\nu}$ is also called the non-isotropic dimension of $\Omega$ related to the vector fields $X$ (cf. \cite{Chen2021,Yung2015}). It follows from \eqref{2-2} that  $n+Q-1\leq \tilde{\nu}< nQ$ for $Q>1$.
\end{definition}

\begin{remark}
\label{remark2-1}
 If the dimension of $V_{j}(x)$ is a constant $\nu_{j}$ in some neighborhood of each $x\in \overline{\Omega}$, then we say the vector fields $X$ satisfy the M\'{e}tivier's condition on $\Omega$. Moreover, the M\'{e}tivier index is defined by
\begin{equation}\label{2-3}
  \nu=\sum_{j=1}^{Q}j(\nu_{j}-\nu_{j-1}),\qquad \nu_{0}:=0,
\end{equation}
which is also called the Hausdorff dimension of $\Omega$ related to the subelliptic metric induced by the vector fields $X$ (cf. \cite{Metivier1976}). Note that $\nu=\tilde{\nu}$ if the M\'{e}tivier's condition is satisfied.\par
 The M\'{e}tivier's condition is also known as the equiregular assumption in sub-Riemannian geometry (cf. \cite{Andrei2019}), which possesses a strong restriction on the vector fields $X$ satisfying H\"ormander's condition. However, there exist many vector fields under H\"ormander's condition but fail to fulfill the M\'{e}tivier's condition (e.g. the Grushin vector fields $X_{1}=\partial_{x_{1}}, X_{2}=x_{1}\partial_{x_{2}}$ in $\mathbb{R}^{2}$), and the generalized M\'{e}tivier index $\tilde{\nu}$ plays a crucial role in the geometry and  functional settings associated with the general H\"{o}rmander vector fields $X$.\par
\end{remark}
\begin{remark}
\label{remark2-1-2}
The generalized M\'{e}tivier index $\tilde{\nu}$ differs from the local homogeneous dimension (denoted by $\tilde{Q}$) on $\overline{\Omega}$ proposed in \cite{Capogna1996,Capogna1994,Capogna1993,Garofalo1996}. Actually, it follows from \eqref{2-1}, \cite[Proposition 2.2]{Chen2019} and \cite[Section 3]{Capogna1996} that $n\leq \nu(x)\leq \max_{x\in\overline{\Omega}}\nu(x)=\tilde{\nu}\leq \tilde{Q}$. In particular, if the vector fields $X$ satisfy the M\'{e}tivier's condition on $\Omega$, then $\tilde{\nu}=\tilde{Q}$.
\end{remark}
The following example shows that the generalized M\'{e}tivier index $\tilde{\nu}$ may strictly less than the local homogeneous dimension $\tilde{Q}$.
\begin{example}\label{ex2-1}
Consider the vector fields $X=(X_1,X_2,X_3)=(e^{x_2}\partial_{x_1}, e^{2x_2}\partial_{x_1}, x_1\partial_{x_2})$ in
$\mathbb{R}^{2}$. Assume that $\Omega=B_1(0)\subset \mathbb{R}^{2}$ is a 2-dimensional unit ball. It follows that the H\"ormander index is 2, and the local
homogeneous dimension $\tilde{Q}=4$. The pointwise homogeneous dimension $\nu(x)=2$ if $x_1\not= 0$, and $\nu(x)=3$ for $x_1=0$. Thus the generalized M\'{e}tivier index $\tilde{\nu}=3<\tilde{Q}$.
\end{example}

Then, we introduce some function spaces (cf. \cite{Xu1990}) associated with the vector fields $X$, which are the natural spaces when dealing with problems related to the H\"{o}rmander operators $\triangle_{X}$.
For vector fields $X=(X_{1},X_{2},\ldots,X_{m})$ and open set $W\subset \mathbb{R}^{n}$, we define
 \[ H_{X}^{1}(W)=\{u\in L^{2}(W)~|~X_{j}u\in L^{2}(W),~ j=1,\ldots,m\}. \]
The function space $H_{X}^{1}(W)$ is analogue to Sobolev space, which
  is a Hilbert space endowed with the norm
\begin{equation}\label{2-4}
 \|u\|^2_{H^{1}_{X}(W)}=\|u\|_{L^2(W)}^2+\|Xu\|_{L^2(W)}^2=\|u\|_{L^2(W)}^2+\sum_{j=1}^{m}\|X_{j}u\|_{L^2(W)}^2.
\end{equation}
In addition, we denote by
   $H^{1}_{X,0}(\Omega)$ the closure of $C_{0}^{\infty}(\Omega)$ in $H_{X}^{1}(W)$. It is well-known that $H^{1}_{X,0}(\Omega)$ is also a Hilbert space endowed with the norm in \eqref{2-4}.

The Poincar\'{e}-Wirtinger type inequality for H\"ormander vector fields was initially studied by Jerison \cite{Jerison1986duke} on the subunit balls with respect to the Carnot-Carath\'{e}odory metric. Subsequently, it was extensively investigated to some special domain (e.g. X-PS domain) by Saloff-Coste \cite{Saloff-Coste1992}, Garofalo-Duy-Minh \cite{Garofalo1996} and Hajlasz-Koskela \cite{Hajlasz2000}, etc. However, when considering the Dirichlet problem on $H_{X,0}^{1}(\Omega)$, we need the following Friedrichs-Poincar\'{e} type inequality for H\"ormander vector fields, which
is different from the Poincar\'{e}-Wirtinger type inequality.

 \begin{proposition}[Friedrichs-Poincar\'{e} type Inequality]
\label{Poincare}
Let $X=(X_{1},X_{2},\ldots,X_{m})$ be the real smooth vector fields defined on an open connected domain
  $W\subset\mathbb{R}^n$, which satisfy the H\"{o}rmander's condition in $W$.  Suppose  $\Omega\subset\subset W$ is a bounded  open domain with smooth boundary. Then the first Dirichlet eigenvalue $\lambda_{1}$ of self-adjoint operator $-\triangle_{X}$ is positive. Moreover,
\begin{equation}\label{2-5}
  \lambda_{1}\int_{\Omega}{|u|^2dx}\leq \int_{\Omega}|Xu|^2dx,~~ \forall u\in H_{X,0}^{1}(\Omega).
  \end{equation}
\end{proposition}
 If there exists at least one vector field $X_{j}~ (1\leq j\leq m)$ which can be globally straightened in $\Omega$, we may prove \eqref{2-5} in a direct way (see \cite[Lemma 5]{Xu1990}). Furthermore, \eqref{2-5} was also achieved in \cite[Lemma 3.2]{Jost1998} and \cite[Proposition 2.1]{Chen2021} under the extra non-characteristic condition on the boundary of $\Omega$. The
characteristic set in the boundary may not be empty for the general smooth domain (e.g. a unit ball in Heisenberg group), which leads to many difficulties in studying the boundary regularity of weak solutions in the characteristic points. However, Derridj \cite{Derridj1972} showed that the characteristic set admits zero measure in the boundary of smooth domain $\Omega$, which allows us to treat \eqref{2-5} without the non-characteristic condition by using the trace theorem locally in the non-characteristic subset of the boundary. Besides, we remark that Bahouri-Chemin-Xu \cite{Bahouri2005} gave a lifting theorem and a trace theorem on non-characteristic surfaces for Sobolev spaces associated with the vector fields $X$ possessing H\"{o}rmander index $2$. The case of non-degenerate  characteristic points in the framework of the Heisenberg group was also studied in \cite{Bahouri2005}.

 We now present the proof of \eqref{2-5} as follows:

 \begin{proof}[Proof of Proposition \ref{Poincare}]
We set
\[ \lambda_{1}=\inf_{\|\varphi\|_{L^{2}(\Omega)}=1,~~ \varphi\in H_{X,0}^{1}(\Omega)}\|X\varphi\|^2_{L^2(\Omega)}. \]
Suppose  $\lambda_{1}=0$. Then there exists a sequence
$\{\varphi_{j}\}_{j=1}^{\infty}$ in  $H_{X,0}^{1}(\Omega)$ such that
$\|X\varphi_{j}\|_{L^2(\Omega)}\to 0$ with
$\|\varphi_{j}\|_{L^2(\Omega)}=1$. Observing that $H_{X,0}^{1}(\Omega)$ is
compactly embedded into $L^2(\Omega)$ (see \cite{Derridj1971}), it follows from \cite[Chapter I, Theorem 1.2]{Struwe2000}
 that there exists
$\varphi_{0}\in H_{X,0}^{1}(\Omega)$ satisfying
$\|\varphi_{0}\|_{L^2(\Omega)}=1$, $\triangle_{X}\varphi_{0}=0$ and $\|X\varphi_{0}\|_{L^2(\Omega)}=
0$. The hypoellipticity of  $\triangle_{X}$ yields $\varphi_{0}\in C^{\infty}(\Omega)$. Moreover, since $X_{j}\varphi_0=0$ in $\Omega$ for $1\leq j\leq m$, we deduce from  H\"{o}rmander's condition that $\partial_{x_{j}}\varphi_0=0$ in $\Omega$ for $1\leq j\leq n$, which means $\varphi_0(x)=C$  in $\Omega$.

For $x_0\in \partial\Omega$, we say $x_0$ is a non-characteristic point of $\partial\Omega$ for vector fields $X_{1},X_{2},\ldots,X_{m}$, if there exists a vector field $X_{j}\in \{X_{1},X_{2},\ldots,X_{m}\}$ such that $X_{j}(x_0)\notin T_{x_{0}}(\partial\Omega)$.  Otherwise, $x_{0}$ is called the characteristic point of vector fields $X$. Let $\partial\Omega=\Gamma_0\cup \Gamma_{1}$, where $\Gamma_0$ is the collection of all non-characteristic points in $\partial\Omega$, and $\Gamma_1=\partial\Omega\setminus \Gamma_0$ denotes the set of characteristic points in $\partial\Omega$. According to Derridj \cite{Derridj1971,Derridj1972},
 we have $|\Gamma_1|_{n-1}=0$ and $\varphi_{0}\in C^{\infty}(\Omega\cup \Gamma_0)$, where $|\Gamma_1|_{n-1}$ denotes the $(n-1)$-dimensional measure of $\Gamma_1$ in  $\partial\Omega$.

We next prove that $\varphi_{0}(x_{0})=0$ holds for some point $x_0\in \Gamma_0$. For any small $\varepsilon>0$, we can choose a non-empty compact set  $\Gamma_{\varepsilon}$ such that $\Gamma_{\varepsilon}\subset \Gamma_{0}\subset \partial\Omega$ and  $|\partial\Omega\setminus \Gamma_{\varepsilon}|_{n-1}\leq \varepsilon$. Owing to \cite[Theorem 2]{Derridj1972},
there is a continuous trace operator $T_{\varepsilon}:H_{X}^{1}(\Omega)\to L^2(\Gamma_{\varepsilon})$ such that $T_{\varepsilon}(u)=u|_{\Gamma_{\varepsilon}}$ for $u\in H_{X}^{1}(\Omega)\cap C^{\infty}(\Omega\cup \Gamma_{\varepsilon})$. For any $u_0\in H_{X,0}^{1}(\Omega)$, there exists a sequence $\{u_{k}\}_{k=1}^{\infty}\subset C_{0}^{\infty}(\Omega)$ such that $u_{k} \to u_{0}$ in $H_{X,0}^{1}(\Omega)$. Observing that $u_{k}\in C_{0}^{\infty}(\Omega)\subset H_{X}^{1}(\Omega)\cap C^{\infty}(\Omega\cup \Gamma_{\varepsilon})$ with $u_{k}|_{\Gamma_{0}}=0$, we have $T_{\varepsilon}(u_{k})\to T_{\varepsilon}(u_0)$ and $T_{\varepsilon}(u_0)=0$. As a result,  $\varphi_{0}\in H_{X,0}^{1}(\Omega)$ implies $T_{\varepsilon}(\varphi_{0})=0$. Moreover, $\varphi_{0}\in C^{\infty}(\Omega\cup \Gamma_\varepsilon)$ due to $\Gamma_{\varepsilon}\subset \Gamma_0$. Thus, we have $\varphi_{0}|_{\Gamma_{\varepsilon}}=T_{\varepsilon}(\varphi_{0})=0$, which indicates $\varphi_{0}(x_{0})=0$ for some
$x_0\in \Gamma_0$. Combining this fact with  $\varphi_{0}\in C^{\infty}(\Omega\cup \Gamma_0)$ and $\varphi_{0}$ is constant in $\Omega$, it follows that $\varphi_{0}(x)=0$ for all $x\in \Omega$. This
contradicts $\|\varphi_{0}\|_{L^2(\Omega)}=1$.
 \end{proof}

The subelliptic Sobolev-type estimates in general have received a lot of attention over the years, e.g. \cite{Capogna1993,Capogna1994,Capogna1996,Hajlasz2000,Cohn2007,Garofalo1996} etc. Roughly speaking, these results obtained the subelliptic Sobolev embedding $H_{X,0}^{1}(\Omega)\hookrightarrow L^{q}(\Omega)$ for $1\leq q\leq \frac{2\tilde{Q}}{\tilde{Q}-2}$, where $\tilde{Q}$ is the local homogeneous dimension of domain $\overline{\Omega}$. In 2015, Yung \cite{Yung2015} gave the following sharp subelliptic Sobolev embedding theorem, which improved the previous results.

 \begin{proposition}[Weighted Sobolev Embedding Theorem]
 \label{pro2-2}
 Let $X$ and $\Omega$ satisfy the assumptions in Proposition \ref{Poincare}. Denote by $\tilde{\nu}$ the generalized M\'{e}tivier index of $\Omega$ defined in \eqref{2-2}. Then for $1\leq p<\tilde{\nu}$ and $q=\frac{\tilde{\nu}p}{\tilde{\nu}-p}$,  there exists a constant $C=C(\Omega,X)>0$ such that
\begin{equation}\label{2-6}
\|u\|_{L^{q}(\Omega)}\leq C\left(\|Xu\|_{L^{p}(\Omega)}+\|u\|_{L^{p}(\Omega)}\right),~~\forall u\in C^{\infty}(\overline{\Omega}).
\end{equation}
\end{proposition}
\begin{proof}
See  \cite [Corollary 1]{Yung2015}.
\end{proof}

\begin{corollary}
\label{corollary2-1}
In particular, for $\tilde{\nu}\geq 3$,  Proposition \ref{pro2-2} implies the embedding
\begin{equation}
\label{2-7}
H_{X,0}^{1}(\Omega)\hookrightarrow L^{q}(\Omega)
\end{equation}
is bounded for any $1\leq q\leq \frac{2\tilde{\nu}}{\tilde{\nu}-2}:=2_{\tilde{\nu}}^{*}$.
\end{corollary}
\begin{remark}
Remark \ref{remark2-1-2} implies that  $\frac{2\tilde{Q}}{\tilde{Q}-2}\leq \frac{2\tilde{\nu}}{\tilde{\nu}-2}$. Hence, the embedding $H_{X,0}^{1}(\Omega)\hookrightarrow L^{\frac{2\tilde{\nu}}{\tilde{\nu}-2}}(\Omega)$ in Corollary \ref{corollary2-1} is sharper than the embedding $H_{X,0}^{1}(\Omega)\hookrightarrow L^{\frac{2\tilde{Q}}{\tilde{Q}-2}}(\Omega)$ in  \cite{Capogna1996,Capogna1994,Capogna1993,Garofalo1996}, where $\frac{2\tilde{\nu}}{\tilde{\nu}-2}$ is the optimal critical exponent for general degenerate elliptic problem.
\end{remark}

In order to prove the Rellich-Kondrachov compact embedding theorem in degenerate case, we need introduce the following abstract version of the Rellich-Kondrachov compactness theorem.

\begin{proposition}\label{prop2-2-1}
Let $Y$ be a set equipped with a finite measure $\mu$. Assume that a linear normed space $G$ of measurable functions on $Y$ has the following two properties:
\begin{enumerate}
  \item [(1)] There exists a constant $q>1$ such that the embedding $G\hookrightarrow L^q(Y,\mu)$ is bounded;
  \item [(2)] Every bounded sequence in $G$ contains a subsequence that converges almost everywhere.
\end{enumerate}
 Then the embedding $G\hookrightarrow L^s(Y,\mu)$ is compact for every $1\leq s<q$.
\end{proposition}
\begin{proof}
See Theorem 4 in \cite{Hajlasz1998}.
\end{proof}

Thus, we have the following compact embedding result.
\begin{proposition}[Degenerate Rellich-Kondrachov Compact Embedding Theorem]
\label{prop2-3}
Let $X$ and $\Omega$ satisfy the assumptions in Proposition  \ref{Poincare}. Then the embedding
\[ H_{X,0}^1(\Omega)\hookrightarrow L^s(\Omega) \]
is compact for every $s\in [1,2_{\tilde{\nu}}^*)$, where $2_{\tilde{\nu}}^*=\frac{2\tilde{\nu}}{\tilde{\nu}-2}$ and $\tilde{\nu}\geq 3$ is the generalized M\'{e}tivier index defined in \eqref{2-2}.
\end{proposition}
\begin{proof}
From \eqref{2-7}, we know the embedding $H_{X,0}^1(\Omega)\hookrightarrow L^{2_{\tilde{\nu}}^*}(\Omega)$ is bounded. For any bounded sequence $\{u_k\}_{k=1}^{\infty}$ in $H_{X,0}^1(\Omega)$, there exists a subsequence $\{u_{k_{j}}\}_{j=1}^{\infty}\subset\{u_k\}_{k=1}^{\infty}$ such that $u_{k_{j}}\rightharpoonup u$ weakly in $H_{X,0}^1(\Omega)$. From subelliptic estimate we know that there exist constants $\epsilon_{0}>0$ and  $C>0$ such that
  $\|u\|_{H^{\epsilon_{0}}(\mathbb{R}^n)}^2\leq C\left(\|Xu\|_{L^2(\mathbb{R}^n)}^{2}+\|u\|_{L^2(\mathbb{R}^n)}^2\right)$ holds for all $u\in H_{X,0}^{1}(\Omega)$. Thus the classical compact embedding $H^{\epsilon_{0}}(\Omega)\hookrightarrow L^{2}(\Omega)$ gives the embedding $H_{X,0}^{1}(\Omega)\hookrightarrow L^{2}(\Omega)$ is compact. That means  $u_{k_{j}}\to u$ in $L^2(\Omega)$. Therefore, the Riesz theorem allows us to find a subsequence $\{v_{j}\}_{j=1}^{\infty}\subset  \{u_{k_{j}}\}_{j=1}^{\infty}$  such that $v_{j}\to u$ almost everywhere on $\Omega$ as $j\to +\infty$. Hence, we conclude from Proposition \ref{prop2-2-1} that the embedding $H_{X,0}^1(\Omega)\hookrightarrow L^s(\Omega)$ is compact for $s\in [1,2_{\tilde{\nu}}^*)$.
 \end{proof}

Combining Proposition \ref{Poincare} and Proposition \ref{prop2-3}, we can prove the well-definedness of subelliptic Dirichlet eigenvalue problem.

\begin{proposition}
\label{prop2-4}
Assume $X$ and $\Omega$ satisfy the assumptions in Proposition \ref{Poincare}.  Then the subelliptic Dirichlet eigenvalue problem
  \begin{equation}\label{2-8}
    \left\{
      \begin{array}{ll}
    -\triangle_{X}u=\lambda u   & \hbox{for~$x\in\Omega$;} \\[2mm]
        u\in H_{X,0}^{1}(\Omega), &
      \end{array}
    \right.
  \end{equation}
  is well-defined, i.e. the self-adjoint operator $-\triangle_{X}$ admits a sequence of discrete Dirichlet eigenvalues $0<\lambda_1<\lambda_2\leq\cdots\leq\lambda_{k-1}\leq\lambda_k\leq\cdots$,
and $\lambda_{k}\to +\infty $ as $k\to +\infty$. Moreover, the corresponding eigenfunctions $\{\varphi_{k}\}_{k=1}^{\infty}$ constitute an orthonormal basis of $L^2(\Omega)$ and also an orthogonal basis of $H_{X,0}^{1}(\Omega)$.
\end{proposition}

We further consider the Dirichlet eigenvalue problem for the subelliptic Schr\"{o}dinger operator $-\triangle_{X}+V$ on $\Omega$,
  \begin{equation}\label{2-9}
    \left\{
      \begin{array}{ll}
    -\triangle_{X}u+V(x)u=\mu u   & \hbox{for~$x\in\Omega$;} \\[2mm]
        u\in H_{X,0}^{1}(\Omega), &
      \end{array}
    \right.
  \end{equation}
 where $V\in L^{\frac{p_{1}}{2}}(\Omega)$ with $p_{1}>\tilde{\nu}$.   Proposition \ref{Poincare} and Proposition \ref{prop2-3} guarantee the validity of eigenvalue problem \eqref{2-9}.
 \begin{proposition}
\label{prop2-5}
 Suppose that $X$ and $\Omega$ satisfy the assumptions in Proposition \ref{Poincare}. If the potential term $V\in L^{\frac{p_{1}}{2}}(\Omega)$ with $p_{1}>\tilde{\nu}$, then the Dirichlet eigenvalue problem \eqref{2-9} of subelliptic Schr\"{o}dinger operator $-\triangle_{X}+V$ is well-defined, i.e. the self-adjoint operator $-\triangle_{X}+V$ admits a sequence of discrete Dirichlet eigenvalues $\mu_{1}\leq \mu_2\leq\cdots\leq\mu_k\leq\cdots$,
and $\mu_{k}\to +\infty $ as $k\to +\infty$.
\end{proposition}
\begin{proof}
 The corresponding quadratic form of $-\triangle_{X}+V$ is given by
 \begin{equation}\label{2-10}
   \mathcal{Q}(u,v)=\int_{\Omega}Xu\cdot Xv dx+\int_{\Omega}V(x)uvdx,\quad \forall u,v\in H_{X,0}^{1}(\Omega).
 \end{equation}
Since $2<\frac{2p_{1}}{p_{1}-2}<\frac{2\tilde{\nu}}{\tilde{\nu}-2}$, by \eqref{2-7} we have for all $u,v\in H_{X,0}^{1}(\Omega)$,
\begin{equation}
\begin{aligned}\label{2-11}
|\mathcal{Q}(u,v)|&\leq \|u\|_{H_{X,0}^{1}(\Omega)}\|v\|_{H_{X,0}^{1}(\Omega)}+\|V\|_{L^{\frac{p_{1}}{2}}(\Omega)}\|u\|_{L^{\frac{2p_{1}}{p_{1}-2}}(\Omega)}\|v\|_{L^{\frac{2p_{1}}{p_{1}-2}}(\Omega)}\\
&\leq \|u\|_{H_{X,0}^{1}(\Omega)}\|v\|_{H_{X,0}^{1}(\Omega)}+\|V\|_{L^{\frac{p_{1}}{2}}(\Omega)}\|u\|_{H_{X,0}^{1}(\Omega)}\|v\|_{H_{X,0}^{1}(\Omega)}.
\end{aligned}
\end{equation}
On the other hand, using H\"{o}lder inequality and Young's inequality we obtain
\begin{equation}
\begin{aligned}\label{2-12}
\mathcal{Q}(u,u)&\geq \frac{\lambda_{1}}{1+\lambda_{1}} \|u\|_{H_{X,0}^{1}(\Omega)}^{2}-\|V\|_{L^{\frac{p_{1}}{2}}(\Omega)}\|u\|_{L^{\frac{2p_{1}}{p_{1}-2}}(\Omega)}^{2}\\
&\geq \frac{\lambda_{1}}{1+\lambda_{1}}\|u\|_{H_{X,0}^{1}(\Omega)}^{2}-\|V\|_{L^{\frac{p_{1}}{2}}(\Omega)}\|u\|_{L^{\frac{2\tilde{\nu}}{\tilde{\nu}-2}}(\Omega)}^{2\frac{\tilde{\nu}}{p_{1}}}\|u\|_{L^{2}(\Omega)}^{2(1-\frac{\tilde{\nu}}{p_{1}})}\\
&\geq \frac{\lambda_{1}}{1+\lambda_{1}}\|u\|_{H_{X,0}^{1}(\Omega)}^{2}-\|V\|_{L^{\frac{p_{1}}{2}}(\Omega)}\left(\varepsilon\|u\|_{L^{\frac{2\tilde{\nu}}{\tilde{\nu}-2}}(\Omega)}^{2}+\varepsilon^{-\frac{\tilde{\nu}}{p_{1}-\tilde{\nu}}}\|u\|_{L^{2}(\Omega)}^{2}\right)
\end{aligned}
\end{equation}
for any $\varepsilon>0$ and any $u\in H_{X,0}^{1}(\Omega)$. Using \eqref{2-7} again and choosing suitable $\varepsilon>0$, we derive from \eqref{2-12} that
\begin{equation}\label{2-13}
  \mathcal{Q}(u,u)\geq \frac{\lambda_{1}}{2(1+\lambda_{1})}\|u\|_{H_{X,0}^{1}(\Omega)}^{2}-C_{0}\|u\|_{L^2(\Omega)}^{2},
\end{equation}
where $C_{0}>0$ is a positive constant depending on $V$.

Consider the quadratic form $\widetilde{\mathcal{Q}}(u,v):=\mathcal{Q}(u,v)+C_{0}(u,v)_{L^2(\Omega)}$ associated with the operator $-\triangle_{X}+V+C_{0}$. It follows from \eqref{2-10}-\eqref{2-13} that $\widetilde{\mathcal{Q}}(u,v)$ is positive, symmetric and closed in $L^2(\Omega)$. Therefore, by the theory of quadratic form (see \cite[Theorem 53.24]{Driver2003}), we know the corresponding operator $\mathcal{L}:=-\triangle_{X}+V+C_{0}$ is a positive self-adjoint operator in $L^2(\Omega)$ with the domain \[\text{dom}(\mathcal{L}):=\{u\in H_{X,0}^{1}(\Omega)|\exists c\geq 0~\mbox{such that}~|\widetilde{\mathcal{Q}}(u,v)|\leq c\|v\|_{L^2(\Omega)}~\forall v\in H_{X,0}^{1}(\Omega)\}. \]
Furthermore,  \eqref{2-13} indicates $\widetilde{\mathcal{Q}}(u,v)$ is coercive. Combining the Riesz representation theorem and Proposition \ref{Poincare}, we can obtain the existence of $\mathcal{L}^{-1}$ and it is also a bounded linear operator
from $L^2(\Omega)$ to $H_{X,0}^{1}(\Omega)$. Using Proposition \ref{prop2-3},  we obtain that $\mathcal{L}^{-1}$ is a compact operator from $L^2(\Omega)$ to $L^2(\Omega)$ since $H_{X,0}^{1}(\Omega)$ can be compactly embedded into $L^2(\Omega)$. From the spectral theory, the subelliptic Schr\"{o}dinger operator $-\triangle_{X}+V$
 has a sequence of discrete Dirichlet eigenvalues $\mu_{1}\leq \mu_2\leq\cdots\leq\mu_{k-1}\leq\mu_k\leq\cdots$,
and $\mu_{k}\to +\infty $ as $k\to +\infty$.
\end{proof}

For the Dirichlet eigenvalue of subelliptic Schr\"{o}dinger operator $-\triangle_{X}+V$, we have the following degenerate
Cwikel-Lieb-Rozenblum inequality.
\begin{proposition}\label{CLR}
Let $p_{1}>\tilde{\nu}$ be a given positive constant and $V\in L^{\frac{p_{1}}{2}}(\Omega)$ be a function satisfying $V\leq 0$, then there exists a positive constant $C>0$ such that
\begin{equation}
  N(0,-\triangle_{X}+V)\leq C\int_{\Omega}|V(x)|^{\frac{\tilde{\nu}}{2}}dx,
\end{equation}
where $N(0,-\triangle_{X}+V):=\#\{k|\mu_{k}<0 \}$ denotes the numbers of negative Dirichlet eigenvalue of $-\triangle_{X}+V$, and $\tilde{\nu}$ is the generalized M\'{e}tivier index of $X$.
\end{proposition}
\begin{proof}
Combining Proposition \ref{Poincare} and Proposition \ref{pro2-2}, we have the following Sobolev inequality
\begin{equation}\label{2-15}
  \left(\int_{\Omega}|u|^{\frac{2\tilde{\nu}}{\tilde{\nu}-2}}dx\right)^{\frac{\tilde{\nu}-2}{2\tilde{\nu}}}\leq C\left(\int_{\Omega}|Xu|^{2}dx\right)^{\frac{1}{2}}~~~\forall u\in H_{X,0}^{1}(\Omega),
\end{equation}
where $C>0$ is a positive constant. Then, we consider the quadratic form
\[ \mathcal{Q}_{1}(u,u)=\int_{\Omega}|Xu|^{2}dx\]
defined on  $H_{X,0}^{1}(\Omega)$.  By
\cite[Remark 7.1]{Chen-Chen-Li2022}, we see that for any
 $u\in H_{X,0}^{1}(\Omega)$,  $|u|\in H_{X,0}^{1}(\Omega)$ and
\begin{equation}\label{2-17}
  \mathcal{Q}_{1}(|u|,|u|)=\int_{\Omega}|X|u||^{2}dx\leq \int_{\Omega}|Xu|^{2}dx=\mathcal{Q}_{1}(u,u).
\end{equation}
Additionally,  we can obtain from \cite[Remark 7.1]{Chen-Chen-Li2022} that if $u\in H_{X,0}^{1}(\Omega)$, then
\begin{equation}\label{2-18}
\min(1,u)=u-\max\{u-1,0\}\in H_{X,0}^{1}(\Omega)
\end{equation}
 and
\begin{equation}\label{2-19}
  \mathcal{Q}_{1}(\min(1,u),\min(1,u))\leq \mathcal{Q}_{1}(u,u).
\end{equation}
Recalling $H_{X,0}^{1}(\Omega)$ is dense in $L^2(\Omega)$ and the corresponding self-adjoint operator $-\triangle_{X}$ is real, \eqref{2-17}-\eqref{2-19} indicate the
Beurling-Deny conditions (cf. \cite[Assumption 2.1]{Frank2010} and \cite{Levin1997}) of $\mathcal{Q}_{1}(u,v)$. Using   \cite[Theorem 2.1]{Frank2010} and \eqref{2-15}, we obtain the degenerate Cwikel-Lieb-Rozenblum inequality
\begin{equation}
  N(0,-\triangle_{X}+V)\leq C\int_{\Omega}|V(x)|^{\frac{\tilde{\nu}}{2}}dx.
\end{equation}

\end{proof}

\subsection{Some estimates in critical point theory}
In this part, we give some estimates in critical point theory.
\begin{proposition}
\label{prop2-7}
Let $\Omega\subset \mathbb{R}^{n}$ be a bounded open domain and $h:\Omega\times \mathbb{R}\to \mathbb{R}$ be a Carath\'{e}odory function satisfying the growth condition:
\[|h(x,u)|\leq a(x)+b|u|^{\frac{q_{1}}{q_{2}}}~~\forall (x,u)\in \Omega\times \mathbb{R},\]
where $b>0$ is a positive constant and $a(x)\geq 0$ with $a(x)\in L^{q_{2}}(\Omega)$. Then the map $u(x)\mapsto h(x,u(x))$ belongs to $ C(L^{q_{1}}(\Omega), L^{q_{2}}(\Omega))$ for any $q_{1}\geq 1$ and $q_{2}\geq 1$.
\end{proposition}
\begin{proof}
See \cite{Krasnosel'skii1964}.
\end{proof}

\begin{proposition}
\label{prop2-8}
Let $X=(X_1,X_{2},\ldots,X_m)$ and $\Omega$ satisfy the conditions in Proposition \ref{Poincare}. Suppose
$ f_{0}: \Omega\times \mathbb{R}\to \mathbb{R}$ is a Carath\'{e}odory function satisfy the growth condition:
\begin{equation}\label{2-21}
  |f_{0}(x,u)|\leq a_{1}(x)+b_{1}|u|^{s_{1}}~~\forall (x,u)\in \Omega\times \mathbb{R},
\end{equation}
where $0\leq s_{1}<\frac{\tilde{\nu}+2}{\tilde{\nu}-2}$, $b_{1}\geq 0$ and $a_{1}(x)\geq 0$ with $a_{1}(x)\in L^{\frac{2\tilde{\nu}}{\tilde{\nu}+2}}(\Omega)$. Denoting by $F_{0}(x,u)=\int_{0}^{u}f_{0}(x,v)dv$, then
$J(u)=\int_{\Omega}F_{0}(x,u(x))dx$ defines a $C^1$-functional on $H_{X,0}^1(\Omega)$, that is to say, the Fr\'{e}chet derivative of $J$ exists and continuous on $H_{X,0}^1(\Omega)$. Furthermore, the Fr\'{e}chet derivative (denoted by $DJ$)  given by
 \[\langle DJ(u),v\rangle=\int_{\Omega}f_{0}(x,u)vdx~~\forall v\in H_{X,0}^1(\Omega)\]
maps the bounded sets in $H_{X,0}^{1}(\Omega)$ to the relatively compact sets in $H_{X}^{-1}(\Omega)$, where $\langle\cdot, \cdot\rangle$ denotes the pairing between $H_{X}^{-1}(\Omega)$ and $H_{X,0}^1(\Omega)$, and $H_{X}^{-1}(\Omega)$ is the dual space of $H_{X,0}^{1}(\Omega)$.
\end{proposition}
\begin{proof}
From \eqref{2-21} we have
\begin{equation}
|F_{0}(x,u)|\leq \left|\int_{0}^{u}|f_{0}(x,v)|dv\right|\leq a_{1}(x)|u|+\frac{b_{1}}{s_{1}+1}|u|^{s_{1}+1}.
\end{equation}
Hence for any $u\in H_{X,0}^{1}(\Omega)$,
\[\begin{aligned}
|J(u)|&\leq \int_{\Omega}|F_{0}(x,u(x))|dx\leq   \int_{\Omega}a_{1}(x)|u|dx+\frac{b_{1}}{s_{1}+1}\int_{\Omega}|u|^{s_{1}+1}dx\\
&\leq \|a_{1}\|_{L^{\frac{2\tilde{\nu}}{\tilde{\nu}+2}}(\Omega)}\|u\|_{L^{\frac{2\tilde{\nu}}{\tilde{\nu}-2}}(\Omega)}+C\|u\|_{H_{X,0}^{1}(\Omega)}^{s_{1}+1}\leq C\|a_{1}\|_{L^{\frac{2\tilde{\nu}}{\tilde{\nu}+2}}(\Omega)}\|u\|_{H_{X,0}^{1}(\Omega)}+C\|u\|_{H_{X,0}^{1}(\Omega)}^{s_{1}+1}.
\end{aligned}\]
That means the functional $J(u)$ is well-defined on $H_{X,0}^{1}(\Omega)$. Besides, we have
\begin{equation}
\begin{aligned}\label{2-23}
&\left|\int_{\Omega}f_{0}(x,u)vdx \right|\leq \int_{\Omega}\left(a_{1}(x)|v|+b_{1}|u|^{s_{1}}|v|\right)dx\\
&\leq \int_{\Omega}\left(a_{1}(x)|v|+C|v|+C|u|^{\frac{\tilde{\nu}+2}{\tilde{\nu}-2}}|v|\right)dx\\
&\leq C\left(\|a_{1}\|_{L^{\frac{2\tilde{\nu}}{\tilde{\nu}+2}}(\Omega)}+1\right)\|v\|_{H_{X,0}^{1}(\Omega)}+C\left(\int_{\Omega}\left( |u|^{\frac{\tilde{\nu}+2}{\tilde{\nu}-2}} \right)^{\frac{2\tilde{\nu}}{\tilde{\nu}+2}} dx\right)^{\frac{\tilde{\nu}+2}{2\tilde{\nu}}} \|v\|_{L^{\frac{2\tilde{\nu}}{\tilde{\nu}-2}}(\Omega)}\\
&\leq C \left(\|a_{1}\|_{L^{\frac{2\tilde{\nu}}{\tilde{\nu}+2}}(\Omega)}+1+\|u\|_{H_{X,0}^{1}(\Omega)}^{\frac{\tilde{\nu}+2}{\tilde{\nu}-2}}\right)\|v\|_{H_{X,0}^{1}(\Omega)},
\end{aligned}
\end{equation}
which implies the G\^{a}teaux derivative $J'(u)$ of functional $J$ at $u\in H_{X,0}^{1}(\Omega)$ given by
 \begin{equation}\label{2-24}
  \langle J'(u),v\rangle=\int_{\Omega}f_{0}(x,u)vdx\qquad \forall v\in H_{X,0}^1(\Omega)
 \end{equation}
is well-defined and belongs to $H_{X}^{-1}(\Omega)$.

On the other hand, for $u,u_0 \in H_{X,0}^1(\Omega)$, we have
\begin{equation}\label{2-25}
 \begin{aligned}
 &\|J'(u)-J'(u_0)\|_{{H_X^{-1}}(\Omega)}=\sup\limits_{
 {v\in H_{X,0}^1(\Omega)},{\|v\|_{{H_{X,0}^1}(\Omega)}\leq1}}|\langle J'(u)-J'(u_0),v\rangle|   \\
 &\leq\sup\limits_{{v\in H_{X,0}^1(\Omega)},
{\|v\|_{{H_{X,0}^1}(\Omega)}\leq1}}\int_{\Omega}|f_{0}(x,u)-f_{0}(x,u_0)||v|dx\\
 &\leq\sup\limits_{
 {v\in H_{X,0}^1(\Omega)}
,{\|v\|_{{H_{X,0}^1}(\Omega)}\leq1}}\left(\int_{\Omega}|f_{0}(x,u)-f_{0}(x,u_0)|^{\frac{2\tilde{\nu}}{\tilde{\nu}+2}}dx\right)^{\frac{\tilde{\nu}+2}{2\tilde{\nu}}}\|v\|_{L^{\frac{2\tilde{\nu}}{\tilde{\nu}-2}}(\Omega)}\\
 &\leq C\sup\limits_{
 {v\in H_{X,0}^1(\Omega)}
,{\|v\|_{{H_{X,0}^1}(\Omega)}\leq1}}\left(\int_{\Omega}|f_{0}(x,u)-f_{0}(x,u_0)|^{\frac{2\tilde{\nu}}{\tilde{\nu}+2}}dx\right)^{\frac{\tilde{\nu}+2}{2\tilde{\nu}}}\|v\|_{H_{X,0}^{1}(\Omega)}\\
 &\leq C\left(\int_{\Omega}|f_{0}(x,u)-f_{0}(x,u_0)|^{\frac{2\tilde{\nu}}{\tilde{\nu}+2}}dx\right)^{\frac{\tilde{\nu}+2}{2\tilde{\nu}}}.
 \end{aligned}
 \end{equation}
It follows from \eqref{2-21} that
\begin{equation}\label{2-26}
  |f_{0}(x,u)|\leq a_{1}(x)+C(1+|u|^{s_{2}}) ~~\forall (x,u)\in \Omega\times \mathbb{R},
\end{equation}
where $s_{2}:=\max\left\{s_{1},\frac{\tilde{\nu}+2}{2\tilde{\nu}}\right\}$ and $C>0$ is a positive constant.
 Thus, Proposition \ref{prop2-7} gives that $u\mapsto f_{0}(\cdot,u(\cdot))$ is continuous from $L^{ \frac{2\tilde{\nu}s_{2}}{\tilde{\nu}+2}}(\Omega)$ to $L^{\frac{2\tilde{\nu}}{\tilde{\nu}+2}}(\Omega)$. Observing that $1\leq s_{2}\cdot \frac{2\tilde{\nu}}{\tilde{\nu}+2}<\frac{2\tilde{\nu}}{\tilde{\nu}-2}$, by Proposition \ref{prop2-3} we know the embedding $H_{X,0}^{1}(\Omega)\hookrightarrow L^{\frac{2\tilde{\nu}s_{2}}{\tilde{\nu}+2}}(\Omega)$ is continuous and  compact. Hence, it follows from \eqref{2-25} that $J'(u)\to J'(u_0)~\mbox{in}~H_X^{-1}(\Omega)$ if $u\to u_{0}$ in $H_{X,0}^{1}(\Omega)$. That means $J$ has a continuous G\^{a}teaux derivative on $H_{X,0}^{1}(\Omega)$. Additionally, we obtain from \cite[Proposition 1.3]{Willem1997} that $J\in C^{1}(H_{X,0}^{1}(\Omega),\mathbb{R})$ and the Fr\'{e}chet derivative $DJ(u)=J'(u)$ for all $u\in H_{X,0}^{1}(\Omega)$.

 We further show that $DJ: u\mapsto DJ(u)$ from $H_{X,0}^1(\Omega)$ to $H_{X}^{-1}(\Omega)$ maps the bounded sets in $H_{X,0}^{1}(\Omega)$ to the relatively compact sets in $H_{X}^{-1}(\Omega)$.  For any bounded sequence $\{u_k\}_{k=1}^{\infty}$ in $H_{X,0}^1(\Omega)$, there exists a subsequence $\{u_{k_j}\}_{j=1}^{\infty}\subset \{u_k\}_{k=1}^{\infty}$ such that $u_{k_j}\rightharpoonup u$ weakly in $H_{X,0}^1(\Omega)$ as $j\to +\infty$. Since the embedding $H_{X,0}^{1}(\Omega)\hookrightarrow L^{\frac{2\tilde{\nu}s_{2}}{\tilde{\nu}+2}}(\Omega)$ is compact, we get $u_{k_{j}}\to u$ in $L^{\frac{2\tilde{\nu}s_{2}}{\tilde{\nu}+2}}(\Omega)$ and $f(\cdot,u_{k_j})\to f_{0}(\cdot,u)$ in $L^{\frac{2\tilde{\nu}}{\tilde{\nu}+2}}(\Omega)$ as $j\to +\infty$.
The estimate \eqref{2-25} yields that $DJ(u_{k_{j}})\to DJ(u)$ in $H_{X}^{-1}(\Omega)$ as $j\to +\infty$. Thus   $DJ$ maps the bounded sets in ${H_{X,0}^1}(\Omega)$ to the relatively compact sets in $H_X^{-1}(\Omega)$.
\end{proof}

Let $V$ be a real Banach space and $V^{*}$ be the dual space of $V$. We recall some basic definitions in critical point theory.
 \begin{definition}
 \label{def2-2}
For $E\in C^{1}(V, \mathbb{R})$, we say a sequence $\{u_m\}_{m=1}^{\infty}$ in $V$ is a Palais-Smale (PS) sequence for functional $E$, if ${|E(u_m)|}\leq C$ uniformly in $m$, and ${\|DE(u_m)\|}_{V^*}\rightarrow0$ as $m\rightarrow\infty$, where $DE:V\to V^{*}$ denotes
the Fr\'{e}chet derivative of $E$.
 \end{definition}

  \begin{definition}
  \label{def2-3}
 A functional $E\in C^{1}(V, \mathbb{R})$ satisfies Palais-Smale condition (henceforth denoted by (PS) condition) if any Palais-Smale sequence has a subsequence which is convergent in $V$.
  \end{definition}

\section{Proof of Theorem \ref{thm1}}
\label{Section3}
In this section, we give the proof of Theorem \ref{thm1}. We first define the weak solution for problem \eqref{problem1-1}.

\begin{definition}
\label{def3-1}
For $f$ and $g$ under the assumptions (H.1)-(H.4), we say  $u\in H_{X,0}^{1}(\Omega)$ is a weak solution of \eqref{problem1-1} if
\begin{equation}
\label{3-1}
  \int_{\Omega}Xu\cdot Xvdx-\int_{\Omega}f(x,u)vdx-\int_{\Omega}g(x,u)vdx=0~~~\forall v\in H_{X,0}^{1}(\Omega).
\end{equation}
\end{definition}
Consider the following energy functional
\begin{equation}\label{3-2}
  E(u):=\frac{1}{2}\int_{\Omega}{{|{Xu}|}^2} dx-\int_{\Omega}F(x,u) dx-\int_{\Omega}G(x,u)dx,
\end{equation}
where $G(x,u)=\int_{0}^{u}g(x,v)dv$ is the primitive of $g(x,u)$.
It follows that
\begin{proposition}
\label{prop3-1}
If the functions $f$ and $g$ satisfy the assumptions (H.1)-(H.4), then the functional
  \[ E(u)=\frac{1}{2}\int_{\Omega}{{|{Xu}|}^2} dx-\int_{\Omega}F(x,u) dx-\int_{\Omega}G(x,u)dx \]
belongs to $C^{1}(H_{X,0}^{1}(\Omega), \mathbb{R})$.  Furthermore, the Fr\'{e}chet derivative of $E$ at $u$ is given by
\begin{equation}\label{3-3}
  \langle DE(u),v\rangle=\int_{\Omega}Xu\cdot Xvdx-\int_{\Omega}f(x,u)vdx-\int_{\Omega}g(x,u)vdx~~~~\forall v\in H_{X,0}^{1}(\Omega).
\end{equation}
 Therefore, the semilinear equation \eqref{problem1-1} is the Euler-Lagrange equation of the variational problem for the energy functional \eqref{3-2}, and the critical point of $E$ in $H_{X,0}^{1}(\Omega)$ is the weak solution to \eqref{problem1-1}.
\end{proposition}
\begin{proof}
First, we can rewrite $E(u)$ as
 \[ \begin{aligned}
 E(u)&=\frac{1}{2}\int_{\Omega}|Xu|^2dx-\int_{\Omega}F(x,u)dx-\int_{\Omega}G(x,u)dx\\
 &=I(u)-J_{1}(u)-J_{2}(u),
 \end{aligned}\]
where
\[ I(u):=\frac{1}{2}\int_{\Omega}|Xu|^2dx,\quad\quad J_{1}(u):=\int_{\Omega}F(x,u)dx\quad\mbox{and}\quad J_{2}(u):=\int_{\Omega}G(x,u)dx.\]
Clearly, the functional $I(u)$ is well-defined on $H_{X,0}^{1}(\Omega)$ and has the G\^{a}teaux derivative $I'(u)$ satisfying
\begin{equation}\label{3-4}
\langle I'(u),v\rangle=\int_{\Omega}Xu\cdot Xvdx \quad\forall v\in H_{X,0}^1(\Omega).
\end{equation}
Observing that for any $u,u_0 \in H_{X,0}^1(\Omega)$,
\begin{align*}
 \|I'(u)-I'(u_0)\|_{{H_X^{-1}}(\Omega)}&=\sup\limits_{
 {v\in H_{X,0}^1(\Omega)},{\|v\|_{{H_{X,0}^1}(\Omega)}\leq1}}|\langle I'(u)-I'(u_0),v\rangle|   \\
 &\leq\sup\limits_{{v\in H_{X,0}^1(\Omega)},
{\|v\|_{{H_{X,0}^1}(\Omega)}\leq1}}\int_{\Omega}|Xu-Xu_0||Xv|dx\\
&\leq\sup\limits_{{v\in H_{X,0}^1(\Omega)},
{\|v\|_{{H_{X,0}^1}(\Omega)}\leq1}}\left(\int_{\Omega}|X(u-u_0)|^2dx\right)^{\frac{1}{2}}\|v\|_{{H_{X,0}^1}(\Omega)}\\
&\leq\|u-u_0\|_{{H_{X,0}^1}(\Omega)}.
 \end{align*}
 Hence $I\in C^{1}(H_{X,0}^{1}(\Omega), \mathbb{R})$ and the Fr\'{e}chet derivative $DI(u)=I'(u)$ for all $u\in H_{X,0}^{1}(\Omega)$. Additionally, Proposition \ref{prop2-8} implies $J_{1},J_{2}\in C^{1}(H_{X,0}^{1}(\Omega), \mathbb{R})$. Consequently, we have $E\in C^{1}(H_{X,0}^{1}(\Omega), \mathbb{R})$ and the Fr\'{e}chet derivative of $E$ at $u$ is given by
\[
  \langle DE(u),v\rangle=\int_{\Omega}Xu\cdot Xvdx-\int_{\Omega}f(x,u)vdx-\int_{\Omega}g(x,u)vdx~~~~\forall v\in H_{X,0}^{1}(\Omega). \]

\end{proof}

Then, we invoke Rabinowitz's perturbation from symmetric method (see \cite{Rabinowitz1982,Rabinowitz1986}) to investigate
 the multiplicity of weak solutions for subelliptic semilinear Dirichlet problem \eqref{problem1-1}.  More precisely, we shall construct a new functional $E_{1}$ which is a modification of $E$ such that the critical values and critical points of $E_1$ coincide with the critical values and critical points of $E$ in high energy level. The conclusion of Theorem \ref{thm1} follows if we prove that $E_{1}$ admits an unbounded sequence of critical values. Before constructing the functional $E_{1}$, we give some necessary estimates.\par

The assumption (H.3) implies for all $|u|\geq R_0$ and all $x\in \overline{\Omega}$,
\[
u|u|^\mu\frac{\partial}{\partial u}(|u|^{-\mu}F(x,u))=f(x,u)u-\mu F(x,u)\geq 0,\]
which gives
\begin{equation}\label{3-5}
  F(x,u)\geq\gamma_0(x)|u|^\mu
\end{equation}
holds for all $|u|\geq R_{0}$ and all $x\in \overline{\Omega}$ with $\gamma_0(x)=R_0^{-\mu}\min{\{F(x,R_0),F(x,-R_0)\}}>0$. Since $F(x,u)\in C(\overline{\Omega}\times\mathbb{R})$, there exists a constant $a_{1}>0$ such that
\begin{equation}\label{3-6}
\gamma_0(x)=R_0^{-\mu}\min{\{F(x,R_0),F(x,-R_0)\}}\geq a_{1}>0~~\forall x\in \overline{\Omega}.
\end{equation}
From \eqref{3-5} and \eqref{3-6},  there is a constant $a_2>0$ such that
\begin{equation}\label{3-7}
F(x,u)\geq a_{1}|u|^\mu-a_2~~~\forall (x,u)\in \overline{\Omega}\times \mathbb{R}.
\end{equation}
Hence, there is a constant $a_3>0$, such that
\begin{equation}\label{3-8}
\frac{1}{\mu}\left(uf(x,u)+a_3\right)\geq F(x,u)+a_2\geq a_1|u|^\mu~~~\forall (x,u)\in \overline{\Omega}\times \mathbb{R}.
\end{equation}

Thus one has
\begin{proposition}
\label{prop3-2}
Under the hypotheses of Proposition \ref{Poincare} and assumptions (H.1)-(H.4), there is a positive constant $A_{0}>0$ such that if $u$ is a critical point if $E$, then
\begin{equation}
\label{3-9}
\int_{\Omega}(F(x,u)+a_2)dx\leq A_{0}(E(u)^2+1)^{\frac{1}{2}}.
\end{equation}
\end{proposition}
\begin{proof}
Since $\alpha(x)\in L^\frac{\mu}{\mu-1}(\Omega)$ and $\mu>\sigma+1$, we have for any $u\in H_{X,0}^{1}(\Omega)$,
\begin{equation}\label{3-10}
  \int_{\Omega}|\alpha(x)u|dx\leq \|u\|_{L^{\mu}(\Omega)}\cdot \|\alpha\|_{L^{\frac{\mu}{\mu-1}}(\Omega)}.
\end{equation}
The Young's inequality gives
\begin{equation}\label{3-11}
  |u|^{\sigma+1}\leq \varepsilon|u|^{\mu}+\varepsilon^{-\frac{\sigma+1}{\mu-\sigma-1}}
\end{equation}
holds for any $\varepsilon>0$.

Suppose $u$ is a critical point of $E$. Using \eqref{3-3} and \eqref{3-8}-\eqref{3-11} we have
\begin{equation}
\label{3-12}
\begin{aligned}
E(u)&=E(u)-\frac{1}{2}\langle DE(u), u\rangle\\
&=\int_{\Omega}\left(\frac{1}{2}uf(x,u)-F(x,u)\right)dx+\int_{\Omega}\left(\frac{1}{2}g(x,u)u-G(x,u)\right)dx\\
&=\frac{1}{2}\int_{\Omega}\left(uf(x,u)+a_{3}\right)dx-\int_{\Omega}(F(x,u)+a_{2})dx+\int_{\Omega}\left(\frac{1}{2}g(x,u)u-G(x,u)\right)dx-a_4\\
&\geq \left(\frac{1}{2}-\frac{1}{\mu}\right)\int_{\Omega}(uf(x,u)+a_{3})dx-\frac{1}{2}\int_{\Omega}|g(x,u)u|dx-\int_{\Omega}|G(x,u)|dx-a_4\\
&\geq a_5 \int_{\Omega}(F(x,u)+a_2)dx-a_{6}\left(\int_{\Omega}|\alpha(x)u|dx+\int_{\Omega}|u|^{\sigma+1}dx\right)-a_{4}\\
&\geq \frac{a_5}{2}\int_{\Omega}(F(x,u)+a_2)dx-a_7,
\end{aligned}
\end{equation}
and \eqref{3-9} follows immediately from \eqref{3-12}.
\end{proof}

Assume $\chi \in C^\infty(\mathbb{R}, \mathbb{R})$ is a smooth function such that $\chi(\xi)\equiv 1$ for $\xi\leq 1$, $\chi(\xi)\equiv 0$ for $\xi\geq 2$, and $\chi'(\xi)\in (-2,0)$ for $\xi \in (1,2)$. Let
\[Q(u):=2A_{0}(E(u)^2+1)^{\frac{1}{2}}\]
and
\[\psi(u):=\chi\left(Q(u)^{-1}\int_{\Omega}(F(x,u)+a_2)dx \right).\]
 The estimate \eqref{3-9} indicates that, if $u$ is a critical point of $E$,  $Q(u)^{-1}\int_{\Omega}(F(x,u)+a_2)dx$ lies in $[0,\frac{1}{2}]$ and then $\psi(u)=1$. Now, we set
\begin{equation}
\label{3-13}
E_1(u):=\frac{1}{2}\int_{\Omega}|Xu|^{2}dx-\int_{\Omega}F(x,u)dx-\psi(u)\int_{\Omega}G(x,u)dx~~~~\forall u\in H_{X,0}^{1}(\Omega).
\end{equation}
It follows from Proposition \ref{prop3-2} that $E_1(u)=E(u)$ if $u$ is a critical point of $E$.\par

We next present several useful lemmas which contain the main technical properties of $E_{1}$.
\begin{lemma}
Under the hypotheses of Proposition \ref{Poincare} and assumptions (H.1)-(H.4), $E_1 \in C^1(H_{X,0}^1(\Omega), \mathbb{R})$.
\end{lemma}
\begin{proof}
From \eqref{3-2} and \eqref{3-13}, we have
\[
 E_{1}(u)=E(u)+(1-\psi(u))\int_{\Omega}G(x,u)dx~~~~\forall u\in H_{X,0}^{1}(\Omega).
\]
Since $\chi$ is smooth, then $\psi\in C^1(H_{X,0}^1(\Omega), \mathbb{R})$ and therefore $E_1\in C^1(H_{X,0}^1(\Omega), \mathbb{R})$.
\end{proof}

\begin{lemma}
\label{lemma3-2}
Under the hypotheses in Proposition \ref{Poincare}  and assumptions (H.1)-(H.4), there exists a positive constant $A$ depending on $g$ such that
  \begin{equation}
  \label{3-14}
  |E_1(u)-E_1(-u)|\leq A(|E_1(u)|^{\frac{\sigma+1}{\mu}}+1)~~~~\forall u\in H_{X,0}^1(\Omega).
  \end{equation}
\end{lemma}

\begin{proof}
If $u\notin {\rm supp}~\psi$, then $\psi(u)=0$ and \eqref{3-14} holds due to assumption (H.1).
Hence, we only need to prove \eqref{3-14} for $u\in {\rm supp}~\psi$.
First,
 for any $u\in {\rm supp}~\psi$, we have
\begin{equation}\label{3-15}
\int_{\Omega}F(x,u)dx\leq 4A_{0}(E(u)^2+1)^{\frac{1}{2}}\leq 4A_{0}(|E(u)|+1).
\end{equation}
  Then by assumption (H.4),\eqref{3-8}, \eqref{3-10} and \eqref{3-15}  we have
   \begin{equation}\label{3-16}
    \begin{aligned}
   \left|\int_{\Omega}G(x,u)dx\right|&\leq \int_{\Omega}\left|\int_{0}^{u(x)}g(x,v)dv\right|dx\\
    &\leq C\|u\|_{L^{\mu}(\Omega)}+\int_{\Omega} C|u|^{\sigma+1} dx\\
   &\leq C+C\left(\int_{\Omega}|u|^\mu dx\right)^{\frac{\sigma+1}{\mu}} dx\\
    &\leq C+C\left(\int_{\Omega}(F(x,u)+a_{2}) dx\right)^{\frac{\sigma+1}{\mu}}\\
    &\leq C(1+|E(u)|^{\frac{\sigma+1}{\mu}}).
    \end{aligned}
  \end{equation}
Besides, since $E(u)=E_{1}(u)+(\psi(u)-1)\int_{\Omega}G(x,u)dx$ and $0\leq \psi\leq 1$, we get
\begin{equation}\label{3-17}
\begin{aligned}
 |E(u)|^{\frac{\sigma+1}{\mu}}&\leq \left(|E_{1}(u)|+\left|\int_{\Omega}G(x,u)dx\right|\right)^{\frac{\sigma+1}{\mu}}\\
 &\leq  |E_{1}(u)|^{\frac{\sigma+1}{\mu}}+\left|\int_{\Omega}G(x,u)dx\right|^{\frac{\sigma+1}{\mu}}.
\end{aligned}
\end{equation}
Recalling that $\frac{\sigma+1}{\mu}<1$,  it follows from \eqref{3-16},  \eqref{3-17} and Young's inequality that
\begin{equation}\label{3-18}
|E(u)|^{\frac{\sigma+1}{\mu}}\leq C\left( |E_{1}(u)|^{\frac{\sigma+1}{\mu}}+1\right).
\end{equation}
Combining \eqref{3-16} and \eqref{3-18},
\begin{equation}\label{3-19}
   \left|\int_{\Omega}G(x,u)dx\right|\leq C(1+|E_{1}(u)|^{\frac{\sigma+1}{\mu}}).
\end{equation}
Observing that  $\psi(u)$ is an even functional, we can deduce from \eqref{3-19} that
  \begin{equation}\label{3-20}
    \begin{aligned}
    |E_1(u)-E_1(-u)|&\leq \psi(u)\left|\int_{\Omega}G(x,u)dx-\int_{\Omega}G(x,-u)dx\right|\\
    &\leq A(1+|E_{1}(u)|^{\frac{\sigma+1}{\mu}}),
    \end{aligned}
  \end{equation}
where $A>0$ is a positive constant depending on $g$.
\end{proof}

Let us analyze the Fr\'{e}chet derivative $DE_{1}$ of $E_{1}$.
From  \eqref{3-13},  for any $u,v\in H_{X,0}^1(\Omega)$,
\begin{equation}\label{3-21}
\langle DE_1(u), v\rangle=\int_{\Omega} Xu\cdot Xvdx-\int_{\Omega}f(x,u)vdx-\langle D\psi(u), v\rangle \int_{\Omega} G(x,u)dx-\psi(u)\int_{\Omega}g(x,u)v dx,
\end{equation}
where
\begin{equation}\label{3-22}
\langle D\psi (u), v\rangle=\chi'(\theta(u))Q(u)^{-2}\left(Q(u)\int_{\Omega}f(x,u)v dx-(2A_{0})^2\theta(u)E(u)\langle DE(u), v\rangle\right)
\end{equation}
and
\begin{equation}\label{3-23}
\theta(u):=Q(u)^{-1}\int_{\Omega}(F(x,u)+a_2) dx.
\end{equation}
Let
\begin{equation}\label{3-24}
T_1(u):=\chi'(\theta(u))(2A_{0})^2Q(u)^{-2}E(u)\theta(u)\int_{\Omega}G(x,u) dx,
\end{equation}
and
\begin{equation}\label{3-25}
T_2(u):=\chi'(\theta(u))Q(u)^{-1}\int_{\Omega}G(x,u) dx+T_1(u).
\end{equation}
Hence,  \eqref{3-21}-\eqref{3-25} give that
\begin{equation}\label{3-26}
\begin{aligned}
\langle DE_1(u), v\rangle=&(1+T_1(u))\int_{\Omega}Xu\cdot Xv dx\\
&-(1+T_2(u))\int_{\Omega}f(x,u)v dx-(\psi(u)+T_1(u))\int_{\Omega}g(x,u)v dx.
\end{aligned}
\end{equation}

Then, we have
\begin{lemma}
\label{lemma3-3}
Under the hypotheses in Proposition \ref{Poincare} and assumptions (H.1)-(H.4),  $T_{1}$ and $T_{2}$ satisfy
\begin{equation}\label{3-27}
\begin{split}
|T_1(u)|&\leq C(|E_{1}(u)|^{\frac{\sigma+1}{\mu}}+1)|E_{1}(u)|^{-1},\\
|T_2(u)|&\leq C(|E_{1}(u)|^{\frac{\sigma+1}{\mu}}+1)|E_{1}(u)|^{-1},
\end{split}
\end{equation}
where $C>0$ is a positive constant.
\end{lemma}
\begin{proof}
If $u\in\rm{supp}~\psi$ with  $\theta(u)<1$, then $\chi'(\theta(u))=0$, which gives $T_{1}(u)=T_{2}(u)=0$ and \eqref{3-27}. For $u\in \rm{supp}~\psi$ with $1\leq \theta(u)\leq 2$, \eqref{3-27} is derived from
 \eqref{3-19}, \eqref{3-24} and \eqref{3-25}.
\end{proof}

\begin{lemma}
\label{lemma3-4}
Under the hypotheses in Proposition \ref{Poincare} and assumptions (H.1)-(H.4),  there is a constant $M_0>0$ such that if $E_1(u)\geq M_0$ and $DE_1(u)=0$, then $E_1(u)=E(u)$ and $DE(u)=0$.
\end{lemma}
\begin{proof}
To prove Lemma \ref{lemma3-4}, it is sufficient to show that if $M_0$ is large and $u$ is a critical point of $E_1$ with $E_1(u)\geq M_0$, then
\begin{equation}\label{3-28}
Q(u)^{-1}\int_{\Omega}(F(x,u)+a_2) dx<1.
\end{equation}
In fact, \eqref{3-28} gives $\theta<1$ and $\psi\equiv1$ in a neighborhood of $u$. That means $\chi'(\theta(u))=0$ and $T_{1}(u)=T_{2}(u)=0$. Then, it follows from \eqref{3-13} and \eqref{3-26} that $E_1(u)=E(u)$ and $DE_{1}(u)=DE(u)$, which yields  Lemma \ref{lemma3-4}.

Let $u\in H_{X,0}^{1}(\Omega)$ be a critical point of $E_{1}$. If $u\in \text{supp}~\psi$ with $\theta(u)<1$, then $T_1(u)=T_2(u)=0$ and $\psi(v)\equiv1 $ in a neighborhood of $u$. Hence, $u$ is also a critical point of $E$ and \eqref{3-9} gives \eqref{3-28}. For $u\notin \text{supp}~\psi$, it follows that $\theta(u)>2$, and $T_1(u)=T_2(u)=0$ due to $\chi'(\theta(u))=0$. Moreover, for sufficiently small $T_{1}(u)$ and $T_{2}(u)$,
\begin{equation}\label{3-29}
\begin{aligned}
E_{1}(u)&=E_1(u)-\frac{1}{2(1+T_1(u))}\langle DE_1(u), u\rangle\\
&=\frac{1+T_2(u)}{2(1+T_1(u))}\int_{\Omega}f(x,u)udx-\int_{\Omega}F(x,u)dx\\
&-\psi(u)\int_{\Omega}G(x,u)dx+\frac{\psi(u)+T_{1}(u)}{2(1+T_{1}(u))}\int_{\Omega}g(x,u)udx.
\end{aligned}
\end{equation}
 Then, we can deduce from \eqref{3-2}, \eqref{3-13} and \eqref{3-29} that
\[ E(u)=E_{1}(u)-\int_{\Omega}G(x,u)dx=\frac{1}{2}\int_{\Omega}f(x,u)udx-\int_{\Omega}F(x,u)dx-\int_{\Omega}G(x,u)dx,\]
which also gives \eqref{3-28} by a similar estimate in \eqref{3-12}. In the case of $u\in \text{supp}~\psi$ with $1\leq \theta(u)\leq 2$, it follows from Lemma \ref{lemma3-3} that there exists a positive constant $M_{0}$ such that for $E_{1}(u)\geq M_{0}$, we have $|T_{1}(u)|\leq \frac{1}{2}$, $|T_{2}(u)|\leq \frac{1}{2}$ and $ \frac{1+T_{2}(u)}{1+T_{1}(u)}>\frac{1}{\mu}+\frac{1}{2}$.
Owing to \eqref{3-2} and \eqref{3-13}, we have
\begin{equation}\label{3-30}
\begin{aligned}
E_1(u)\leq |E_1(u)|&\leq |E(u)|+(1-\psi(u))\left|\int_{\Omega}G(x,u)dx\right|\leq |E(u)|+\left|\int_{\Omega}G(x,u)dx\right|.
\end{aligned}
\end{equation}
On the other hand, by \eqref{3-8} and \eqref{3-29}, we have
\begin{equation}\label{3-31}
\begin{aligned}
E_1(u)&\geq \left(\frac{1+T_2(u)}{2(1+T_1(u))}-\frac{1}{\mu}\right)\int_{\Omega}(uf(x,u)+a_{3}) dx\\
&-C(u)\int_{\Omega}|g(x,u)u|dx-\psi(u)\left|\int_{\Omega}G(x,u)dx\right|-a_8,
\end{aligned}
\end{equation}
where $C(u)=\left|\frac{\psi(u)+T_{1}(u)}{2(1+T_{1}(u))}\right|\leq 2$. Combining \eqref{3-30} and \eqref{3-31}, we obtain
\begin{equation}\label{3-32}
|E(u)|\geq \frac{1}{2}\left(\frac{1}{2}-\frac{1}{\mu}\right)\int_{\Omega}(uf(x,u)+a_{3}) dx-2\int_{\Omega}|g(x,u)u|dx-2\int_{\Omega}|G(x,u)|dx-a_8.
\end{equation}
According to \eqref{3-32} and Young's inequality,  we can also deduce \eqref{3-28} by similar approach of \eqref{3-12} with $A_{0}$ replaced by a larger constant which is smaller than $2A_{0}$.
\end{proof}

\begin{lemma}
\label{lemma3-5}
Under the hypotheses in Proposition \ref{Poincare} and assumptions (H.1)-(H.4), there is a constant $M_1\geq M_0$ such that $E_1$ satisfies $(PS)$ condition on $\widehat{A}_{M_{1}}:=\{u\in H_{X,0}^{1}(\Omega)|E_{1}(u)\geq M_{1}\}$, where $M_{0}>0$ is the positive constant appeared in Lemma \ref{lemma3-4}.
\end{lemma}

\begin{proof}
It is sufficient to show there exists a positive constant $M_1> M_0$ such that if the sequence $\{u_m\}_{m=1}^{\infty}\subset H_{X,0}^1(\Omega)$ satisfies $M_1\leq E_1(u_m)\leq K$ and $DE_1(u_m)\to 0$ as $m\to \infty$, then $\{u_m\}_{m=1}^{\infty}$ is bounded in $H_{X,0}^1(\Omega)$ and admits a convergent subsequence.

For sufficiently large $m$ (such that $\|DE_1(u_m)\|_{H_{X}^{-1}(\Omega)}<1$) and any $\rho>0$, it derives from \eqref{2-5}, \eqref{3-13} and \eqref{3-26} that
\begin{equation}\label{3-33}
\begin{aligned}
K+\rho\|u_m\|_{H_{X,0}^1(\Omega)}&\geq E_1(u_m)-\rho\langle DE_1(u_m), u_m\rangle\\
&\geq \frac{\lambda_1}{1+\lambda_1}\left(\frac{1}{2}-\rho(1+T_1(u_m))\right)\|u_m\|^2_{H_{X,0}^1(\Omega)}\\
&+\rho(1+T_2(u_m))\int_{\Omega}f(x,u_m)u_m dx-\int_{\Omega}F(x,u_m)dx\\
&+\rho(\psi(u_m)+T_1(u_m))\int_{\Omega}g(x,u_{m})u_m dx-\psi(u_{m})\int_{\Omega}G(x,u_{m})dx.
\end{aligned}
\end{equation}
On the other hand, by Lemma \ref{lemma3-3} we can find a positive constant $M_1\geq M_{0}$ such that $|T_{1}(u)|, |T_{2}(u)|\leq \frac{1}{2}$ and $ \frac{1+T_{2}(u)}{1+T_{1}(u)}>\frac{1}{\mu}+\frac{1}{2}>\frac{2}{\mu}$ on $\widehat{A}_{M_{1}}=\{u\in H_{X,0}^{1}(\Omega)|E_{1}(u)\geq M_{1}\}$. Now, taking
 $\rho>0$ and $\varepsilon>0$ such that
\begin{equation}\label{3-34}
\frac{1}{2(1+T_1(u_m))}> \rho +\varepsilon > \rho-\varepsilon > \frac{1}{\mu(1+T_2(u_m))},
\end{equation}
we can deduce from  \eqref{3-8}, \eqref{3-33}, \eqref{3-34} and assumption (H.3) that
\begin{eqnarray}\label{3-35}
&&K+\rho\|u_m\|_{H_{X,0}^1(\Omega)}\nonumber\\
&&\geq
\frac{\varepsilon\lambda_1}{1+\lambda_1}(1+T_1(u_m))\|u_m\|^2_{H_{X,0}^1(\Omega)}+\left(\frac{1}{\mu}+\varepsilon(1+T_2(u_m))\right)\int_{|u_m|\geq R_0}f(x,u_m)u_m dx\nonumber\\
&&+\rho(1+T_2(u_m))\int_{|u_m|\leq R_0}f(x,u_m)u_m dx-\int_{\Omega}F(x,u_m)dx\nonumber\\
&&+\rho(\psi(u_m)+T_1(u_m))\int_{\Omega}g(x,u_{m})u_m dx-\psi(u_{m})\int_{\Omega}G(x,u_{m})dx\nonumber\\
&&\geq \frac{\varepsilon\lambda_1}{2(1+\lambda_1)}\|u_m\|^2_{H_{X,0}^1(\Omega)}+\left(\frac{1}{\mu}+\frac{\varepsilon}{2}\right)\int_{|u_m|\geq R_0}f(x,u_m)u_m dx\\
&&-\int_{|u_m|\geq R_0}F(x,u_m)dx-\frac{3}{2}\int_{\Omega}|g(x,u_{m})u_{m}|dx-\int_{\Omega}|G(x,u_{m})|dx-C\nonumber\\
&&\geq \frac{\varepsilon\lambda_1\|u_m\|^2_{H_{X,0}^1(\Omega)}}{2(1+\lambda_1)}+\frac{\varepsilon \mu}{2}\int_{\Omega}(F(x,u_m)+a_{2})dx-\frac{3}{2}\int_{\Omega}|g(x,u_{m})u_{m}|dx-\int_{\Omega}|G(x,u_{m})|dx-C\nonumber\\
&&\geq \frac{\varepsilon\lambda_1}{2(1+\lambda_1)}\|u_m\|^2_{H_{X,0}^1(\Omega)}+\frac{\varepsilon \mu}{4}\int_{\Omega}(F(x,u_m)+a_{2})dx-C\nonumber\\
&&\geq \frac{\varepsilon\lambda_1}{2(1+\lambda_1)}\|u_m\|^2_{H_{X,0}^1(\Omega)}-C,\nonumber
\end{eqnarray}
which yields $\{u_m\}_{m=1}^{\infty}$ is bounded in $H_{X,0}^1(\Omega)$.

We next prove $\{u_m\}_{m=1}^{\infty}$ has a convergent subsequence in  $H_{X,0}^1(\Omega)$.
Consider the  quadratic form
\begin{equation}\label{3-36}
a[u,v]=\int_{\Omega}Xu\cdot Xvdx,~~~\forall u,v\in H_{X,0}^1(\Omega).
\end{equation}
Clearly,
\begin{equation}\label{3-37}
|a[u,v]|\leq \|u\|_{H_{X,0}^{1}(\Omega)}\cdot \|v\|_{H_{X,0}^{1}(\Omega)}.
\end{equation}
Therefore for any $u\in  H_{X,0}^{1}(\Omega)$,  $a[u,\cdot]\in H_{X}^{-1}(\Omega)$ determines a functional $L(u)\in H_{X}^{-1}(\Omega)$ satisfying
\begin{equation}\label{3-38}
  \langle L(u),v\rangle=a[u,v]~~~\forall v\in H_{X,0}^{1}(\Omega).
\end{equation}
It follows from \eqref{3-36} and \eqref{3-38} that $L: u\mapsto L(u)$ is a linear operator from  $H_{X,0}^{1}(\Omega)$ to $H_{X}^{-1}(\Omega)$. Besides, \eqref{3-37} gives
\begin{equation}\label{3-39}
\|L(u)\|_{H_{X}^{-1}(\Omega)}=\sup_{v\in H_{X,0}^{1}(\Omega), \|v\|_{H_{X,0}^{1}(\Omega)}\leq 1}|\langle L(u),v\rangle|\leq \|u\|_{H_{X,0}^{1}(\Omega)},
\end{equation}
which implies $L$ is a bounded linear operator from  $H_{X,0}^{1}(\Omega)$ to $H_{X}^{-1}(\Omega)$. Owing to Proposition \ref{Poincare}, we have
\begin{equation}\label{3-40}
  a[u,u]=\int_{\Omega}|Xu|^{2}dx\geq \frac{\lambda_{1}}{1+\lambda_{1}}\|u\|_{H_{X,0}^{1}(\Omega)}^{2}~~\mbox{for all}~u\in H_{X,0}^{1}(\Omega).
\end{equation}
Combining \eqref{3-37}, \eqref{3-40} and Lax-Milgram theorem,  $L:{H_{X,0}^1}(\Omega)\to {H_{X}^{-1}(\Omega)}$ is a bounded invertible linear map.\par

On the other hand, for any $u\in H_{X,0}^{1}(\Omega)$, we deduce from Proposition \ref{prop2-8} that the
linear functionals $K_{1}(u),K_{2}(u)$ given by
\begin{equation}\label{3-41}
  \langle K_{1}(u),v\rangle:=-\int_{\Omega}f(x,u)vdx\quad\mbox{and}\quad \langle K_{2}(u),v\rangle:=-\int_{\Omega}g(x,u)vdx,~~\forall v\in H_{X,0}^{1}(\Omega),
\end{equation}
belong to $H_{X}^{-1}(\Omega)$. Moreover,  $K_{1},K_{2}:H_{X,0}^{1}(\Omega)\to H_{X}^{-1}(\Omega)$ map the bounded sets in $H_{X,0}^{1}(\Omega)$ to the relatively compact sets in $H_{X}^{-1}(\Omega)$. It follows from \eqref{3-26} that the Fr\'{e}chet derivative of $E_{1}$ can be decomposed into
\begin{equation}\label{3-42}
 DE_{1}(u)=(1+T_1(u))L(u)+(1+T_2(u))K_{1}(u)+(\psi(u)+T_1(u))K_{2}(u)
 \end{equation}
for all $u\in H_{X,0}^{1}(\Omega)$. That means
\begin{equation}\label{3-43}
L^{-1}DE_{1}(u_m)=(1+T_1(u_m))u_m+(1+T_2(u_m))L^{-1}K_{1}(u_m)+(\psi(u_m)+T_1(u_m))L^{-1}K_{2}(u_{m}).
\end{equation}
Furthermore, since $\{T_1(u_m)\}_{m=1}^{\infty}$, $\{T_2(u_m)\}_{m=1}^{\infty}$ and $\{\psi(u_m)\}_{m=1}^{\infty}$ are bounded, there exists a subsequence $\{u_{m_k}\}_{k=1}^{\infty}\subset \{u_{m}\}_{m=1}^{\infty}$ such that
\begin{equation}\label{3-44}
\lim_{k\rightarrow\infty}T_1(u_{m_k})=\overline{a_1},~~\lim_{k\rightarrow\infty}T_2(u_{m_k})=\overline{a_2},~~\mbox{and}~~~\lim_{k\rightarrow\infty}\psi(u_{m_k})=\overline{a_3}.
\end{equation}
Observing  $\{u_{m_{k}}\}_{k=1}^{\infty}$ is also bounded in $H_{X,0}^1(\Omega)$, $L^{-1}K_{1}(u_{m_{k}})$ and $L^{-1}K_{2}(u_{m_{k}})$ converge along a subsequence $\{u_{m_{k_{j}}}\}_{j=1}^{\infty}\subset \{u_{m_k}\}_{k=1}^{\infty}$. As a result of \eqref{3-43} and \eqref{3-44}, we conclude that $\{u_{m_{k_j}}\}_{j=1}^{\infty}$ converges in $H_{X,0}^1(\Omega)$.
\end{proof}

From Lemma \ref{lemma3-4}, we see that  Theorem \ref{thm1} can be achieved by showing $E_1$ admits an unbounded sequence of critical points. We will proceed with this goal in several steps. First, we introduce a sequence of minimax values of $E_1$.

For any finite dimensional subspace $W\subset H_{X,0}^1(\Omega)$, if $u\in W$ such that $\|u\|_{H_{X,0}^{1}(\Omega)}=\rho>0$, we let $v=\frac{u}{\rho}$. Then $\|v\|_{H_{X,0}^{1}(\Omega)}=1$, and we obtain from assumption (H.3) and \eqref{3-6} that
\begin{equation}\label{3-45}
\begin{aligned}
E_{1}(u)&=E_{1}(\rho v)\\
&=\frac{\rho^2}{2}\int_{\Omega}|Xv|^2dx-\int_{\Omega}F(x,\rho v)dx-\psi(\rho v)\int_{\Omega}G(x,\rho v)dx\\
&\leq\frac{\rho^2}{2}-a_{1}\rho^{\mu}\int_{|\rho v|\geq R_0}|v|^{\mu}dx+|\Omega|\sup\limits_{x\in{\overline{\Omega}},|w|\leq R_0}|F(x,w)|+\int_{\Omega}|G(x,\rho v)|dx\\
&\leq\frac{\rho^2}{2}-a_{1}\rho^\mu\int_{|\rho v|\geq R_0}|v|^\mu dx+|\Omega|\sup\limits_{x\in{\overline{\Omega}},|w|\leq R_0}|F(x,w)|\\
&+\rho\int_{\Omega}|\alpha(x)v|dx+\frac{\beta}{\sigma+1}\rho^{\sigma+1}\int_{\Omega}|v|^{\sigma+1}dx\to -\infty, ~\mbox{as}~\rho\rightarrow {+\infty}.
\end{aligned}
\end{equation}
  The estimate \eqref{3-45}  indicates that for any finite dimensional subspace $W\subset H_{X,0}^{1}(\Omega)$, there is a constant $R=R(W)>0$ such that $E_{1}(u)< 0$ for all $u \in W$ with $ \|u\|_{H_{X,0}^{1}(\Omega)}\geq R$. In particular, for each $j\geq 1$, we let $W_j:=\text{span}\{\varphi_k|1\leq k\leq j\}$ and $W_j^{\perp}:=\text{span}\{\varphi_k|k\geq j+1\}$ be the orthogonal complement of $W_j$ in $H_{X,0}^1(\Omega)$, where $\varphi_{k}$ is the $k$-th Dirichlet eigenfunction of $-\triangle_{X}$.
 We can choose an increasing positive sequence $\{R_j\}_{j=1}^{\infty}$ such that $$R_j\geq \lambda_j^{\frac{r}{2(p-2)}}$$ and $E_1(u)< 0$ for all $u\in W_j$ with $\|u\|_{H_{X,0}^{1}(\Omega)}\geq R_j$, where $r=\tilde{\nu}(1-\frac{p}{2_{\tilde{\nu}}^*})$. \par

For each $j\in \mathbb{N}^{+}$, we let $D_j:=B_{R_j}\cap W_j$ and
\[G_j:=\{h\in C(D_j,H_{X,0}^{1}(\Omega))| ~h ~\mbox{is odd and} ~h=\mbox{\textbf{id} on }\partial B_{R_j}\cap W_j \},\]
where $B_{R}=\{u\in H_{X,0}^{1}(\Omega)|\|u\|_{H_{X,0}^{1}(\Omega)}\leq R\}$ is the closed ball of radius $R$ in $H_{X,0}^{1}(\Omega)$ and $\textbf{id}$ denotes the identity map. Clearly, $\textbf{id}\in G_{j}$. Besides, for each $k\in \mathbb{N}^{+}$, we define
\begin{equation}\label{3-46}
b_k:=\inf\limits_{h\in G_k}\max\limits_{u\in D_k}E_1(h(u)),
\end{equation}
\[U_k:=\{u=t\varphi_{k+1}+w|~t\in [0, R_{k+1}], w\in B_{R_{k+1}}\cap W_k, \|u\|_{H_{X,0}^{1}(\Omega)}\leq R_{k+1}\},\]
\[  Q_k:=(\partial B_{R_{k+1}}\cap W_{k+1})\cup ((B_{R_{k+1}}\setminus B_{R_{k}})\cap W_k) \]
and
\begin{align*}
\Lambda_k:=\{H\in C(U_k, H_{X,0}^{1}(\Omega))|~H|_{D_k}\in G_k ~\mbox{and}~H(u)=u~\mbox{if}~u\in Q_k\cap U_{k}\}.
\end{align*}
We also set
\begin{equation}\label{3-47}
c_k:=\inf\limits_{H\in \Lambda_k}\max\limits_{u\in U_k}E_1(H(u)).
\end{equation}
Observing that $D_{k}\subset U_{k}$ and $H|_{D_{k}}\in G_{k}$ for any $H\in \Lambda_k$, it follows that
\[ \max_{u\in U_{k}}E_{1}(H(u))\geq \max_{u\in D_{k}}E_{1}(H(u))=\max_{u\in D_{k}}E_{1}(H|_{D_{k}}(u))\geq \inf_{h\in G_{k}}\max_{u\in D_{k}}E_{1}(h(u)), \]
which means $c_k\geq b_k$.

\begin{proposition}
\label{prop3-3}
Assume $c_k>b_k\geq M_1$. For $\delta\in (0, c_k-b_k)$, we denote by
\[\Lambda_k(\delta):=\{H\in \Lambda_k|~E_1(H(u))\leq b_k+\delta ~\mbox{for}~u\in D_k\}\]
and
\[c_k(\delta):=\inf\limits_{H\in \Lambda_k(\delta)}\max\limits_{u\in U_k}E_1(H(u)).\]
Then $c_k(\delta)$ is a critical value of $E_1$.

\end{proposition}

\begin{proof}
We first show that $\Lambda_k(\delta)\neq \varnothing$. The definition of $b_k$ implies for any $\delta\in (0, c_k-b_k)$, there exists a map $h\in G_k$ such that $E_1(h(u))\leq b_k+\delta$ holds for all $u\in D_k$. Let
\begin{equation}\label{3-48}
 H_{0}(u):=\left\{\begin{array}{ll}{u,} & { u\in Q_k\cap U_{k}} \\[2mm] {h(u),} & {u\in D_k.}\end{array}\right.
\end{equation}
It follows that $H_{0}\in C((Q_k\cap U_{k})\cup D_{k},H_{X,0}^{1}(\Omega))$. Since $(Q_k\cap U_{k})\cup D_{k}$ is closed in $U_{k}$, by Dugundji's extension theorem (see \cite[Theorem 4.1]{Dugundji1951}) we can find a map $\widetilde{H}\in C(U_k, H_{X,0}^{1}(\Omega))$ such that $\widetilde{H}(u)=H_{0}(u)$ for all $u\in (Q_k\cap U_{k})\cup D_{k}$. Clearly, $\widetilde{H} \in \Lambda_k(\delta)$ and therefore $\Lambda_k(\delta)\neq \varnothing$.\par

Observe that $\Lambda_{k}(\delta)\subset \Lambda_{k}$ gives $c_{k}(\delta)\geq c_{k}$. Suppose $c_k(\delta)$ is not a critical value of $E_1$. Set $\bar{\varepsilon}:=\frac{1}{2}(c_k-b_k-\delta)>0$.
By Lemma \ref{lemma3-5} and Deformation Theorem (see \cite[Lemma 1.60]{Rabinowitz1982} and \cite[Theorem A.4]{Rabinowitz1986}), for $c_k(\delta)>M_{1}$ and $\bar{\varepsilon}>0$, there exist  $\varepsilon\in (0, \bar{\varepsilon})$ and $\eta\in C([0,1]\times H_{X,0}^{1}(\Omega), H_{X,0}^{1}(\Omega))$ such that
\begin{equation}\label{3-49}
\eta(t,u)=u ~~~\forall t\in [0,1]~ ~\mbox{if}~~E_1(u)\notin [c_k(\delta)-\bar{\varepsilon}, c_k(\delta)+\bar{\varepsilon}],
\end{equation}
and
\begin{equation}\label{3-50}
\eta(1, A_{c_k(\delta)+\varepsilon})\subset A_{c_k(\delta)-\varepsilon},
\end{equation}
where $A_c:=\{u\in H_{X,0}^{1}(\Omega)|~E_1(u)\leq c\}$.
According to the definition of $c_k(\delta)$, we can choose a $H\in \Lambda_k(\delta)$ such that
\begin{equation}\label{3-51}
\max\limits_{u\in U_k}E_1(H(u))\leq c_k(\delta)+\varepsilon.
\end{equation}

We claim that $\eta (1, H(\cdot))\in \Lambda_k(\delta)$. Clearly, $\eta (1, H(\cdot))\in C(U_k, H_{X,0}^{1}(\Omega))$.
Since $H\in \Lambda_k$, if $u\in Q_k\cap U_{k}$, $H(u)=u$ and therefore $E_1(H(u))=E_1(u)\leq 0$ according to the definitions of $R_k$ and $R_{k+1}$. Recalling that $c_k(\delta)\geq c_k>b_k\geq M_1>0$, we have $$c_{k}(\delta)-\bar{\varepsilon}=c_{k}(\delta)-\frac{1}{2}c_{k}+\frac{1}{2}b_{k}+\frac{1}{2}\delta>0,$$ which means \[ E_1(H(u))\notin [c_{k}(\delta)-\bar{\varepsilon}, c_{k}(\delta)+\bar{\varepsilon}]\qquad \forall u\in Q_k\cap U_{k}.\]
 Owing to \eqref{3-49}, we get  $\eta(1, H(u))=H(u)=u$ on $ Q_k\cap U_{k}$. Then, we show that $\eta(1,H(\cdot))|_{D_{k}}\in G_{k}$.
For any $u\in D_k$, $E_1(H(u))\leq b_k+\delta$ due to $H\in \Lambda_k(\delta)$ and therefore
\[E_1(H(u))-\frac{1}{2}(b_k+\delta)\leq \frac{1}{2}(b_k+\delta)<\frac{1}{2}c_k\leq c_k(\delta)-\frac{1}{2}c_k,\]
which yields
\begin{equation*}
E_1(H(u))< c_k(\delta)-\frac{1}{2}(c_k-b_k-\delta)=c_k(\delta)-\bar{\varepsilon}.
\end{equation*}
Using \eqref{3-49} again,  $\eta(1, H(u))=H(u)$ for all $u\in D_k$. Similarly, we can also get $\eta(1, H(-u))=H(-u)=-H(u)$ for all $u\in D_{k}$. Therefore, $\eta(1, H(\cdot))$ is odd on $ D_k$. Moreover, for any $u\in \partial B_{R_k}\cap W_k$, we have $H(u)=u$ and $E_{1}(u)\leq 0$, which implies $\eta(1, H(u))=u$ on $\partial B_{R_k}\cap W_k$.  Hence,  $\eta (1, H(\cdot))|_{D_k}\in G_k$ and $\eta (1, H(\cdot))\in \Lambda_k$. Furthermore, the arguments above indicate that, for $u\in D_k$, $E_1(\eta(1, H(u)))=E_1(H(u))\leq b_k+\delta$. Consequently, we obtain $\eta (1, H(\cdot))\in \Lambda_k(\delta)$.
Besides, from \eqref{3-51} we know  $H(u)\in A_{c_{k}(\delta)+\varepsilon}$ for all $u\in U_{k}$. Thus, \eqref{3-50} yields that
\[\max_{u\in U_k}E_1(\eta(1, H(u)))\leq c_k(\delta)-\varepsilon,\]
which contradicts the definition of $c_k(\delta)$.
\end{proof}

Next, we give the lower bound of $b_{k}$ by using the condition $(L)$ in Theorem \ref{thm1}.

\begin{proposition}
\label{prop3-4}
Under the assumptions in Theorem \ref{thm1}, there exist constants $C_{1}>0$ and $\tilde{k}\in \mathbb{N}^{+}$ such that
\begin{equation}\label{3-52}
  b_{k}\geq C_{1}\cdot k^{\frac{2}{\vartheta}\left(\frac{p}{p-2}-\frac{\tilde{\nu}}{2}\right)}(\ln k)^{-\kappa\left(\frac{p}{p-2}-\frac{\tilde{\nu}}{2}\right)}\qquad \forall k\geq \tilde{k}.
\end{equation}
\end{proposition}

\begin{proof}
For $h \in G_k$ and $\rho<R_k$, by the intersection Theorem (see \cite[Lemma 1.44]{Rabinowitz1982}) we have $h(D_k)\cap \partial B_{\rho}\cap W_{k-1}^{\bot}\neq\varnothing$. Then

\begin{equation}\label{3-53}
\begin{aligned}
\max\limits_{u\in D_k}E_1(h(u))&=\max\limits_{u\in h(D_k)}E_1(u)\geq \max\limits_{u\in h(D_k)\cap \partial B_{\rho}\cap W_{k-1}^{\bot}}E_1(u)\\
&\geq \inf\limits_{u\in h(D_k)\cap \partial B_{\rho}\cap W_{k-1}^{\bot} }E_1(u)\geq\inf\limits_{u\in\partial B_{\rho}\cap W_{k-1}^{\bot} }E_1(u).
\end{aligned}
\end{equation}
Moreover, for any $u\in W_{k-1}^{\bot}\cap\partial B_{\rho}$, we deduce from Rayleigh–Ritz formula that
\begin{equation}\label{3-54}
\begin{aligned}
E_1(u)&=\frac{1}{2}\int_{\Omega}|Xu|^2dx-\int_{\Omega}F(x,u)dx-\psi(u)\int_{\Omega}G(x,u) dx\\
&\geq \frac{1}{2}\int_{\Omega}|Xu|^2dx-C\int_\Omega|u|^pdx-C\int_\Omega|u|dx-\left(\int_{\Omega}|\alpha(x)u(x)|dx+\frac{\beta}{\sigma+1}\int_{\Omega}|u|^{\sigma+1}dx\right)\\
      &\geq \frac{\lambda_1}{2(1+\lambda_1)}\|u\|_{H_{X,0}^{1}(\Omega)}^2-C\int_{\Omega}|u|^{p}dx-\|\alpha\|_{L^{\frac{2\tilde{\nu}}{\tilde{\nu}+2}}(\Omega)}\|u\|_{L^{\frac{2\tilde{\nu}}{\tilde{\nu}-2}}(\Omega)}-C\\
       &\geq \frac{\lambda_1}{2(1+\lambda_1)}\|u\|_{H_{X,0}^{1}(\Omega)}^2-C\int_{\Omega}|u|^{p}dx-C\|\alpha\|_{L^{\frac{2\tilde{\nu}}{\tilde{\nu}+2}}(\Omega)}\|u\|_{H_{X,0}^{1}(\Omega)}-C\\
       &\geq \frac{\lambda_1}{4(1+\lambda_1)}\|u\|_{H_{X,0}^{1}(\Omega)}^2-C\|u\|_{L^2(\Omega)}^r\|u\|_{L^{2_{\tilde{\nu}}^*}(\Omega)}^{p-r}-C\\
       &\geq \frac{\lambda_1}{4(1+\lambda_1)}\|u\|_{H_{X,0}^{1}(\Omega)}^2-C\lambda_{k}^{-\frac{r}{2}}\|u\|_{H_{X,0}^{1}(\Omega)}^p-C\\
       &=\left(\frac{\lambda_1}{4(1+\lambda_1)}-C\lambda_{k}^{-\frac{r}{2}}\rho^{p-2}   \right)\rho^2-C,
\end{aligned}
\end{equation}
where $r$ is a positive constant such that $\frac{r}{2}+\frac{p-r}{2_{\tilde{\nu}}^*}=1$ (i.e.  $r=\tilde{\nu}\left(1-\frac{p}{2_{\tilde{\nu}}^*}\right)$). Taking $\rho=\varepsilon_0\lambda_k^{\frac{r}{2(p-2)}}$ in \eqref{3-54} with $0<\varepsilon_0<\min\left\{1,\left(\frac{\lambda_1}{8C(1+\lambda_1)}\right)^{\frac{1}{p-2}}\right\}$, we get
\begin{equation}\label{3-55}
E_1(u)\geq C_{0}\lambda_{k}^{\frac{r}{p-2}}=C_{0}\lambda_k^{\left(\frac{p}{p-2}-\frac{\tilde{\nu}}{2}\right)}~~~\forall k\geq \hat{k},
\end{equation}
where $\hat{k}$ is some positive integer  and $C_{0}>0$ is a constant.
Thus, combining \eqref{3-46}, \eqref{3-53} and \eqref{3-55}, we get from condition $(L)$ that
\begin{equation}\label{3-56}
  b_{k}\geq C_{0}\lambda_k^{\left(\frac{p}{p-2}-\frac{\tilde{\nu}}{2}\right)}\geq  C_{1}\cdot k^{\frac{2}{\vartheta}\left(\frac{p}{p-2}-\frac{\tilde{\nu}}{2}\right)}(\ln k)^{-\kappa\left(\frac{p}{p-2}-\frac{\tilde{\nu}}{2}\right)}\qquad \forall k\geq \tilde{k},
\end{equation}
where $C_{1}>0$ is a positive constant and $\tilde{k}\geq \hat{k}$ is a positive integer.
\end{proof}

Finally, we show that $c_k>b_k$ holds for infinitely
many integers $k\in \mathbb{N}^{+}$.
\begin{proposition}
\label{prop3-5}
Let $\tilde{k}$ be the same constant given in \eqref{3-52}, for any positive integer $\bar{k}$ satisfying $\bar{k}\geq\tilde{k}$, if $c_k=b_k$ for all $k\geq \bar{k}$, then there exists a constant $\widetilde{M}>0$  such that
\begin{equation}\label{3-57}
b_k\leq \widetilde{M}~~~\forall k\geq 1.
\end{equation}
Hence, from Proposition \ref{prop3-4}, we know that $c_k>b_k$ must hold for infinitely
many $k$.

\end{proposition}

\begin{proof}
For any $\varepsilon>0$ and $k\geq \bar{k}$, the definition of $c_{k}$ allows us to choose a $H\in \Lambda_k$ such that
\begin{equation}\label{3-58}
\max\limits_{u\in U_k}E_1(H(u))\leq c_k+\varepsilon=b_k+\varepsilon.
\end{equation}
Observe that $D_{k+1}=U_k\cup (-U_k)=B_{R_{k+1}}\cap W_{k+1}$ is a compact set in finite dimensional space $W_{k+1}$. Hence, $H$ can be continuously extended to $D_{k+1}$ as an odd map which belongs to $G_{k+1}$. Therefore by \eqref{3-46}
\begin{equation}\label{3-59}
  b_{k+1}=\inf_{h\in G_{k+1}}\max_{u\in D_{k+1}}E_{1}(h(u))\leq \max_{u\in D_{k+1}}E_{1}(H(u))=E_{1}(H(v_{0}))
\end{equation}
for some $v_{0}\in D_{k+1}$ depending on $k$. If $v_{0}\in U_{k}$, by \eqref{3-58} and \eqref{3-59},
\begin{equation}\label{3-60}
  b_{k+1}\leq  E_{1}(H(v_{0}))\leq b_{k}+\varepsilon.
\end{equation}
Suppose $v_{0}\in -U_{k}$. Then, by Lemma \ref{lemma3-2} we obtain
\begin{equation}\label{3-61}
  E_{1}(-H(v_{0}))\geq E_{1}(H(v_{0}))-A(|E_{1}(H(v_{0}))|^{\frac{\sigma+1}{\mu}}+1).
\end{equation}
Since $b_{k}\to \infty$ as $k\to +\infty$ and $\sigma+1<\mu$,  \eqref{3-59} and \eqref{3-61} give that
 $ E_{1}(-H(v_{0}))=E_{1}(H(-v_{0}))>0$ for large $k$ (e.g. $k\geq k_{1}$ with $k_1\geq \bar{k}$). Using \eqref{3-60} and Lemma \ref{lemma3-2} again, we get
\begin{equation}\label{3-62}
\begin{aligned}
  E_{1}(H(v_{0}))&=E_{1}(-H(-v_{0}))\leq E_{1}(H(-v_{0}))+A\left((E_{1}(H(-v_{0})))^{\frac{\sigma+1}{\mu}}+1\right)\\
  &\leq b_{k}+\varepsilon+A[(b_{k}+\varepsilon)^{\frac{\sigma+1}{\mu}}+1]~~~~\forall k\geq k_{1}.
\end{aligned}
\end{equation}
Therefore, we conclude from \eqref{3-59}, \eqref{3-60} and \eqref{3-62} that
\begin{equation}\label{3-63}
  b_{k+1}\leq b_{k}+\varepsilon+A[(b_{k}+\varepsilon)^{\frac{\sigma+1}{\mu}}+1]~~~~\forall k\geq k_{1}.
\end{equation}
Since $\varepsilon$ is arbitrary, \eqref{3-63} derives
\begin{equation}\label{3-64}
  b_{k+1}\leq b_{k}+A(b_{k}^{\frac{\sigma+1}{\mu}}+1)\leq b_{k}+2Ab_{k}^{\frac{\sigma+1}{\mu}}=b_{k}\left(1+2Ab_{k}^{\frac{\sigma+1-\mu}{\mu}}\right)~~~~\forall k\geq k_{2},
\end{equation}
where $k_{2}\geq k_{1}$ such that $b_{k}\geq 1$ for $k\geq k_{2}$. By iteration we obtain  for any $l\in \mathbb{N}^{+}$,
\begin{equation}\label{3-65}
\begin{aligned}
 b_{k_{2}+l}&\leq   b_{k_{2}}\prod_{k=k_{2}}^{k_{2}+l-1}\left(1+2Ab_{k}^{\frac{\sigma+1-\mu}{\mu}}\right)\\
 &= b_{k_{2}}\exp\left(\sum_{k=k_{2}}^{k_{2}+l-1}\log\left(1+2Ab_{k}^{\frac{\sigma+1-\mu}{\mu}} \right) \right)\\
 &\leq b_{k_{2}}\exp\left(2A \sum_{k=k_{2}}^{k_{2}+l-1} b_{k}^{\frac{\sigma+1-\mu}{\mu}}\right).
\end{aligned}
\end{equation}

On the other hand, using Proposition \ref{prop3-4} we have
\[ b_{k}\geq C_{1}\cdot k^{\frac{2}{\vartheta}\left(\frac{p}{p-2}-\frac{\tilde{\nu}}{2}\right)}(\ln k)^{-\kappa\left(\frac{p}{p-2}-\frac{\tilde{\nu}}{2}\right)}~~~~~\forall k\geq k_{2}. \]
The condition $(A1)$ gives
 \[ \frac{2p}{\vartheta(p-2)}-\frac{\tilde{\nu}}{\vartheta}>\frac{\mu}{\mu-\sigma-1},\]
which means
\[ \frac{2}{\vartheta}\left(\frac{p}{p-2}-\frac{\tilde{\nu}}{2}\right)\frac{\sigma+1-\mu}{\mu}<-1. \]
Hence, there exists a positive constant $M_{2}$ such that
\begin{equation}\label{3-66}
  \sum_{k=k_{2}}^{\infty} b_{k}^{\frac{\sigma+1-\mu}{\mu}}\leq C_{1}^{\frac{\sigma+1-\mu}{\mu}}\sum_{k=k_{2}}^{\infty} k^{\frac{2}{\vartheta}\left(\frac{p}{p-2}-\frac{\tilde{\nu}}{2}\right)\frac{\sigma+1-\mu}{\mu}}(\ln k)^{\kappa\left( \frac{p}{p-2}-\frac{\tilde{\nu}}{2}\right)\frac{\mu-\sigma-1}{\mu}} \leq M_{2}<+\infty.
\end{equation}
Therefore, by \eqref{3-65} and \eqref{3-66}, we can find a positive constant $\widetilde{M}$ such that $b_{k}\leq \widetilde{M}$ for all $k\geq 1$.
\end{proof}

Now, let us finish the proof of Theorem \ref{thm1}.

\begin{proof}[Proof of Theorem \ref{thm1}]
From Proposition \ref{prop3-4}, there exists $\bar{k}_{1}\geq \tilde{k}$ such that $b_{k}\geq M_{1}\geq M_{0}$ for all $k\geq \bar{k}_{1}$. Owing to Proposition \ref{prop3-4} and Proposition \ref{prop3-5}, we can find an integer $l_{1}\geq \bar{k}_{1}$ such that $c_{l_{1}}>b_{l_{1}}$. Then, it follows from Lemma \ref{lemma3-4} and Proposition \ref{prop3-3} that there exists $\delta_{1}\in (0,c_{l_{1}}-b_{l_{1}})$ such that $c_{l_{1}}(\delta_{1})\geq c_{l_{1}}>b_{l_{1}}$ is a critical value of $E_{1}$ and is also a critical value of $E$. Next, using Proposition \ref{prop3-4} again, we can find a $\bar{k}_{2}\geq l_{1}$ such that $b_{k}\geq c_{l_{1}}(\delta_{1})+1$ for all $k\geq \bar{k}_{2}$. Similarly,
Proposition  \ref{prop3-4} and Proposition  \ref{prop3-5} imply that there exists $l_{2}\geq \bar{k}_{2}$ such that $c_{l_{2}}>b_{l_{2}}$. Moreover, Lemma \ref{lemma3-4} and Proposition  \ref{prop3-3} indicate there exists $\delta_{2}\in (0, c_{l_{2}}-b_{l_{2}})$ such that $c_{l_{2}}(\delta_{2})\geq c_{l_{2}}>b_{l_{2}}\geq c_{l_{1}}(\delta_{1})+1$ is  another critical value of $E$. Repeating this process, we can deduce that the functional $E$ possesses infinitely many critical points $\{u_m\}_{m=1}^{\infty}$ in $H_{X,0}^1(\Omega)$ such that
\begin{equation}\label{3-67}
 E(u_m)=c_{l_{m}}(\delta_{m}),\quad 0<c_{l_{1}}(\delta_{1})<c_{l_{2}}(\delta_{2})<\cdots<c_{l_{m}}(\delta_{m})<\cdots,
\end{equation}
and $c_{l_{m}}(\delta_{m})\to +\infty$ as $m\to+\infty$.
Besides, we can deduce from \eqref{3-2} and \eqref{3-16} that, there exists a positive constant $\widehat{C}>0$, such that
\begin{equation}\label{3-68}
|E(u)|\leq \widehat{C}\left(\|u\|_{H_{X,0}^{1}(\Omega)}+\|u\|_{H_{X,0}^{1}(\Omega)}^{\sigma+1}+\|u\|_{H_{X,0}^{1}(\Omega)}^{p} \right)~~~~\forall u\in H_{X,0}^{1}(\Omega).
\end{equation}
Hence,  \eqref{3-67}, \eqref{3-68} and Proposition \ref{prop3-1} imply that $\{u_{m}\}_{m=1}^{\infty}$ is also an unbounded sequence of weak solutions in $H_{X,0}^{1}(\Omega)$.
\end{proof}

\section{Proof of Theorem \ref{thm2}}
\label{Section4}
In this section, we shall prove Theorem \ref{thm2} by the arguments involving Morse index and degenerate Cwikel-Lieb-Rozenblum inequality.

For $2<p<2_{\tilde{\nu}}^{*}=\frac{2\tilde{\nu}}{\tilde{\nu}-2}$, we define
\begin{equation}\label{4-1}
I_p(u):=\frac{1}{2}\int_{\Omega}|Xu|^{2}dx-\frac{B}{p}\int_{\Omega}|u|^pdx~~~~\forall u\in H_{X,0}^{1}(\Omega),
\end{equation}
where $B>0$ is a positive constant which will be determined later. Clearly, we can deduce from Proposition \ref{prop3-1} that $I_p\in C^{1}(H_{X,0}^{1}(\Omega),\mathbb{R})$  by
 taking $g(x,u)=0$ and $f(x,u)=B|u|^{p-2}u$. Furthermore, we have
\begin{proposition}
\label{prop4-1}
 $I_p\in  C^{2}(H_{X,0}^{1}(\Omega),\mathbb{R})$ and
\begin{equation}\label{4-2}
 \langle I''_{p}(u)h,v\rangle=\int_{\Omega}Xh\cdot Xvdx-(p-1)B\int_{\Omega}|u|^{p-2}hv dx,
\end{equation}
where $I''_{p}$ is the second order G\^{a}teaux derivative of $I_{p}$.
\end{proposition}
\begin{proof}
Since $I_p\in C^{1}(H_{X,0}^{1}(\Omega),\mathbb{R})$, we have
\begin{equation}\label{4-3}
  \langle DI_p(u), v\rangle= \langle I'_p(u), v\rangle=\int_{\Omega}Xu\cdot Xv dx-B\int_{\Omega}|u|^{p-2}uvdx~~\forall u,v\in H_{X,0}^{1}(\Omega),
\end{equation}
where $DI_p$ is the Fr\'{e}chet derivative and $I'_{p}$ is G\^{a}teaux derivative. For any $u,v,h\in H_{X,0}^{1}(\Omega)$ and any $0<|t|<1$, we have
\begin{equation}\label{4-4}
\frac{1}{t}\langle I'_{p}(u+th)-I'_{p}(u),v\rangle=\int_{\Omega}Xh\cdot Xvdx-\frac{B}{t}\int_{\Omega}\left(|u+th|^{p-2}(u+th)-|u|^{p-2}u\right)vdx.
\end{equation}
Observing there exists $\lambda\in (0,1)$ such that
\begin{equation}\label{4-5}
\begin{aligned}
 \left|\left(\frac{|u+th|^{p-2}(u+th)-|u|^{p-2}u}{t}\right)v\right|&=(p-1)|u+\lambda th|^{p-2}|h||v|\\
 &\leq (p-1)\left(|u|+|h|\right)^{p-2}|h||v|.
\end{aligned}
\end{equation}
Because $u,v,h\in H_{X,0}^{1}(\Omega)\subset L^{p}(\Omega)$, the H\"{o}lder inequality implies
\begin{equation}\label{4-6}
  \int_{\Omega}|u+h|^{p-2}|h||v| dx\leq \|u+h\|_{L^{p}(\Omega)}^{p-2}\|h\|_{L^{p}(\Omega)}\|v\|_{L^{p}(\Omega)}.
\end{equation}
Using Lebesgue's dominated convergence theorem, we have
\begin{equation}\label{4-7}
\begin{aligned}
 \langle I''_{p}(u)h,v\rangle&=\lim_{t\to 0}\frac{1}{t}\langle I'_{p}(u+th)-I'_{p}(u),v\rangle\\
 &=\int_{\Omega}Xh\cdot Xvdx-(p-1)B\int_{\Omega}|u|^{p-2}hv dx~~~~\forall u,v,h\in H_{X,0}^{1}(\Omega),
\end{aligned}
\end{equation}
where $I''_{p}$ is the second order G\^{a}teaux derivative of $I_{p}$.

We next show the continuity of $ I''_{p}$. Suppose that $\{u_{n}\}_{n=1}^{\infty}$ is a sequence in $H_{X,0}^{1}(\Omega)$ such that $u_{n}\to u_{0}$ in $H_{X,0}^{1}(\Omega)$ as $n\to +\infty$. Corollary \ref{corollary2-1} gives that $u_{n}\to u_{0}$ in $L^{p}(\Omega)$ as $n\to +\infty$. Then, it derives from Proposition \ref{prop2-7} that $|u_{n}|^{p-2}\to |u_{0}|^{p-2}$ in $L^{\frac{p}{p-2}}(\Omega)$ as $n\to +\infty$. Therefore, for any $h,v\in H_{X,0}^{1}(\Omega)\subset L^{p}(\Omega)$, a direct calculation gives that
\begin{equation}\label{4-8}
\begin{aligned}
\left|\langle (I''_{p}(u_{n})-I''_{p}(u_{0}))h,v\rangle\right|&\leq (p-1)B\int_{\Omega}\left||u_{n}|^{p-2}-|u_{0}|^{p-2} \right||h||v|dx\\
&\leq (p-1)B\left(\int_{\Omega}\left||u_{n}|^{p-2}-|u_{0}|^{p-2} \right|^{\frac{p}{p-2}}dx\right)^{\frac{p-2}{p}}\|h\|_{L^{p}(\Omega)}\|v\|_{L^{p}(\Omega)}\\
&\to 0~~~\mbox{as}~~~n\to+\infty,
\end{aligned}
\end{equation}
which means $I_p\in  C^{2}(H_{X,0}^{1}(\Omega),\mathbb{R})$.
\end{proof}

On the other hand, since $DI_p(u)\in H_{X}^{-1}(\Omega)$ for each $u\in H_{X,0}^{1}(\Omega)$, the Riesz representation theorem gives
\begin{equation}\label{4-9}
  \langle DI_p(u), v\rangle=(\widetilde{I_p'}(u),v)_{H_{X,0}^{1}(\Omega)}~~~~\forall v\in H_{X,0}^{1}(\Omega),
\end{equation}
where $\widetilde{I_p'}(u)\in H_{X,0}^{1}(\Omega)$ is uniquely determined by $DI_p(u)$, and $(u,v)_{H_{X,0}^{1}(\Omega)}=\int_{\Omega}Xu\cdot Xv dx+\int_{\Omega}uv dx$ is the inner in $H_{X,0}^{1}(\Omega)$. We have the following compactness result of $\widetilde{I_{p}'}(u)$.

\begin{proposition}
\label{prop4-2}
The operator $\widetilde{I_p'}:H_{X,0}^{1}(\Omega)\to H_{X,0}^{1}(\Omega)$ given by \eqref{4-9} satisfies
\begin{equation}\label{4-10}
 \widetilde{I_p'}(u)=u+K_{3}(u)\qquad \forall u\in H_{X,0}^{1}(\Omega),
\end{equation}
where the operator $K_{3}:H_{X,0}^{1}(\Omega)\to H_{X,0}^{1}(\Omega)$ maps the bounded sets to the relatively compact sets.
\end{proposition}

\begin{proof}
  Since
  \[\begin{aligned}
  (\widetilde{I_p'}(u),v)_{H_{X,0}^{1}(\Omega)}
  &=\int_{\Omega}Xu\cdot Xvdx-B\int_{\Omega}|u|^{p-2}uvdx\\
  &=(u,v)_{H_{X,0}^{1}(\Omega)}-B\int_{\Omega}|u|^{p-2}uvdx-\int_{\Omega}uvdx ~~~~\forall v\in H_{X,0}^{1}(\Omega),
  \end{aligned}  \]
  we  have
\[ \widetilde{I_p'}(u)=u+K_{3}(u)~~~~\forall u\in H_{X,0}^{1}(\Omega),\]
  where $K_{3}(u)$ satisfies
\[ (K_{3}(u),v)_{H_{X,0}^{1}(\Omega)}:=-B\int_{\Omega}|u|^{p-2}uvdx-\int_{\Omega}uvdx~~~~\forall v\in H_{X,0}^{1}(\Omega).\]
Moreover, it follows from Proposition \ref{prop2-8} that the operator $K_{3}:H_{X,0}^{1}(\Omega)\to H_{X,0}^{1}(\Omega)$ maps the bounded sets to the relatively compact sets.
\end{proof}
Additionally, the functional $I_{p}$ satisfies the following Palais-Smale type properties proposed in \cite{Tanaka1989}.
\begin{proposition}
\label{prop4-3}
If for some $M>0$, the sequence $\{u_{j}\}_{j=1}^{\infty}\subset H_{X,0}^{1}(\Omega)$ satisfies $I_{p}(u_{j})\leq M$ for all $j\geq 1$ and $\|DI_{p}(u_{j})\|_{H_{X}^{-1}(\Omega)}\to 0$ as $j\to \infty$, then $\{u_{j}\}_{j=1}^{\infty}$ has a subsequence which is convergent
in $H_{X,0}^{1}(\Omega)$.
\end{proposition}

\begin{proof}
By the similar calculations in \eqref{3-33} and \eqref{3-35} above, we can deduce that $\{u_j\}_{j=1}^{\infty}$ is bounded in $H_{X,0}^{1}(\Omega)$. According to Proposition \ref{prop4-2}, we have
\begin{equation}
u_j=\widetilde{I_p'}(u_j)-K_{3}(u_j),
\end{equation}
where $K_{3}:H_{X,0}^{1}(\Omega)\to H_{X,0}^{1}(\Omega)$ maps the bounded sets to the relatively compact sets.
Besides, \eqref{4-9} indicates that $\|\widetilde{I_p'}(u_j)\|_{H_{X,0}^{1}(\Omega)}\to 0$ as $j\to+\infty$. Therefore, we can find a subsequence of $\{u_{j}\}_{j=1}^{\infty}$ which is convergent  in $H_{X,0}^{1}(\Omega)$.
  \end{proof}

\begin{proposition}
\label{prop4-4}
For $m\geq 1$, let $V_{m}$ be the $m$-dimensional subspace of $H_{X,0}^{1}(\Omega)$ and $V_{m}^{*}$ be the dual space of $V_{m}$. If the sequence $\{u_{j}\}_{j=1}^{\infty}\subset V_{m}$ satisfies
\begin{enumerate}
  \item [(1)]$I_{p}(u_{j})\leq M$ for all $j\geq 1$ and some $M>0$;
  \item  [(2)] $\|D(I_{p}|_{V_{m}})(u_{j})\|_{V_{m}^{*}}\to 0$ as $j\to \infty$,
\end{enumerate}
then $\{u_{j}\}_{j=1}^{\infty}$ has a subsequence which is convergent
in $V_{m}$.
\end{proposition}
\begin{proof}
For any $v\in V_{m}$, we have
\[ \langle D(I_{p}|_{V_{m}})(u_{j}),v\rangle=\langle DI_{p}(u_{j}),v\rangle, \]
which gives
\[ |\langle DI_{p}(u_{j}),u_{j}\rangle|\leq \|D(I_{p}|_{V_{m}})(u_{j})\|_{V_{m}^{*}}\|u_{j}\|_{V_{m}}=\|D(I_{p}|_{V_{m}})(u_{j})\|_{V_{m}^{*}}\|u_{j}\|_{H_{X,0}^{1}(\Omega)}.\]
Combining this fact and using the similar calculations in \eqref{3-33} and \eqref{3-35} above, we can also deduce that $\{u_j\}_{j=1}^{\infty}$ is bounded in $V_{m}$. Since $V_{m}$ is a finite dimensional space, it follows that $\{u_{j}\}_{j=1}^{\infty}$ possesses a subsequence which is convergent
in $V_{m}$.
\end{proof}

\begin{proposition}
\label{prop4-5}
For each $j\geq 1$, let $V_{j}$ be the $j$-dimensional subspace of $H_{X,0}^{1}(\Omega)$ and $V_{j}^{*}$ be the dual space of $V_{j}$. If the sequence $\{u_{j}\}_{j=1}^{\infty}$ satisfies
\begin{enumerate}
  \item [(1)] $u_{j}\in V_{j}$ for all $j\geq 1$;
  \item [(2)] $I_{p}(u_{j})\leq M$ for all $j\geq 1$ and some $M>0$;
  \item [(3)] $\|D(I_{p}|_{V_{j}})(u_{j})\|_{V_{j}^{*}}\to 0$ as $j\to \infty$,
\end{enumerate}
then $\{u_{j}\}_{j=1}^{\infty}$ has a subsequence which is convergent
in $H_{X,0}^{1}(\Omega)$.
\end{proposition}
\begin{proof}
Observing that for any $v\in V_{j}$, we have
\begin{equation}\label{4-12}
 \langle D(I_{p}|_{V_{j}})(u_{j}),v\rangle=\langle DI_{p}(u_{j}),v\rangle,
\end{equation}
which means
\begin{equation}\label{4-13}
  |\langle DI_{p}(u_{j}),u_{j}\rangle|\leq \|D(I_{p}|_{V_{j}})(u_{j})\|_{V_{j}^{*}}\|u_{j}\|_{V_{j}}=\|D(I_{p}|_{V_{j}})(u_{j})\|_{V_{j}^{*}}\|u_{j}\|_{H_{X,0}^{1}(\Omega)}.
\end{equation}
Using \eqref{4-12} and the similar calculations in \eqref{3-33} and \eqref{3-35} above, it follows that $\{u_j\}_{j=1}^{\infty}$ is bounded in $H_{X,0}^{1}(\Omega)$. On the other hand, \eqref{4-9} gives
\begin{equation}\label{4-14}
\begin{aligned}
 \langle DI_{p}(u_{j}),v\rangle&=(\widetilde{I_p'}(u_j),v)_{H_{X,0}^{1}(\Omega)}\\
 &=(\widetilde{I_p'}(u_j),P_{j}v)_{H_{X,0}^{1}(\Omega)}\\
 &=(P_{j}\widetilde{I_p'}(u_j),v)_{H_{X,0}^{1}(\Omega)}\quad\forall v\in V_{j}\subset H_{X,0}^{1}(\Omega),
\end{aligned}
\end{equation}
where $P_{j}$ is the projection from $H_{X,0}^{1}(\Omega)$ to $V_j$. Now, substituting $v=P_{j}\widetilde{I_p'}(u_j)\in V_{j}$ in \eqref{4-14}. It derives from \eqref{4-12} and \eqref{4-14} that
\begin{equation}\label{4-15}
\begin{aligned}
  (P_{j}\widetilde{I_p'}(u_j),P_{j}\widetilde{I_p'}(u_j))_{H_{X,0}^{1}(\Omega)}&=|\langle DI_{p}(u_{j}),P_{j}\widetilde{I_p'}(u_j)\rangle|\\
  &=|\langle D(I_{p}|_{V_{j}})(u_{j}),P_{j}\widetilde{I_p'}(u_j)\rangle|\\
  &\leq \|D(I_{p}|_{V_{j}})(u_{j})\|_{V_{j}^{*}}\|P_{j}\widetilde{I_p'}(u_j)\|_{H_{X,0}^{1}(\Omega)},
\end{aligned}
\end{equation}
which implies $P_{j}\widetilde{I_p'}(u_j)\to 0$ in $H_{X,0}^{1}(\Omega)$ as $j\to +\infty$. Furthermore, Proposition \ref{prop4-2} yields
\[
u_j=\widetilde{I_p'}(u_j)-K_{3}(u_j),\]
where $K_{3}:H_{X,0}^{1}(\Omega)\to H_{X,0}^{1}(\Omega)$ maps the bounded sets to the relatively compact sets. That means
\begin{equation}\label{4-16}
  u_j=P_{j}\widetilde{I_p'}(u_j)-P_{j}K_{3}(u_j)~~~~\forall j\geq 1.
\end{equation}
The boundedness of $\{u_{n}\}_{n=1}^{\infty
}$ allows us to find a subsequence $\{v_{j}\}_{j=1}^{\infty}\subset \{u_{n}\}_{n=1}^{\infty
}$ such that $v_{j}\rightharpoonup v_{0}$ weakly in $H_{X,0}^{1}(\Omega)$ as $j\to +\infty$. Therefore, $P_{j}K_{3}(v_j)\to  K_{3}(v_0)$ in $H_{X,0}^{1}(\Omega)$ as $j\to +\infty$, which leads
\[ v_{j}=P_{j}\widetilde{I_p'}(v_j)-P_{j}K_{3}(v_j)\to -K_{3}(v_0) ~~\mbox{in}~~~H_{X,0}^{1}(\Omega)~~\mbox{as}~~~j\to+\infty. \]

\end{proof}

We next define the Morse index and augmented Morse index on the critical points of $I_{p}$.
\begin{definition}
 Let $u\in H_{X,0}^{1}(\Omega)$ be a critical point of $I_{p}$. Denote by
 $V$ the subspace of $ H_{X,0}^{1}(\Omega)$.
The augmented Morse index $m^*(u)$ at $u$ is defined by
 \[m^*(u):=\max\{{\rm dim} V|~ \langle I''_{p}(u)v,v\rangle\leq 0,~~\forall v\in V\}.\]
In particular,  the  Morse index $m(u)$ at $u$ is given by
 \[m(u):=\max\{{\rm dim} V|~ \langle I''_{p}(u)v,v\rangle< 0,~~\forall v\in V\setminus\{0\}\}.\]

\end{definition}

The augmented Morse indices admit the following estimate.
\begin{proposition}
\label{prop4-6}
For each $k\geq 1$, let $e_{k}:=\inf\limits_{h\in G_k}\max\limits_{u\in D_k}I_{p}(h(u))$. Then there exists $v_k\in H_{X,0}^{1}(\Omega)$ such that
\begin{equation}\label{4-17}
 I_p(v_k)\leq e_{k},\qquad DI_p(v_k)=0.
\end{equation}
Moreover
\begin{equation}\label{4-18}
 m^*(v_k)\geq k~~~~~\forall k\geq 1,
\end{equation}
where $m^*(v_k)$ is the augmented Morse index at $v_k$ related to $I_{p}$.

\end{proposition}
\begin{proof}
According to \eqref{4-1}, we can verify that $I_{p}(0)=0$ and $I_{p}(-u)=I_{p}(u)$ for all $u\in H_{X,0}^{1}(\Omega)$.
 Propositions \ref{prop4-1}-\ref{prop4-5} and \cite[Theorem B]{Tanaka1989} yield \eqref{4-17} and \eqref{4-18}.
\end{proof}
Owing to Proposition \ref{prop2-5}, Proposition \ref{CLR} and Proposition \ref{prop4-6}, we can derive the following lower bound of $b_{k}$.

\begin{proposition}
\label{prop4-7}
For the minimax values $b_{k}$, $e_{k}$  given in \eqref{3-46} and Proposition \ref{prop4-6}, we have
\begin{equation}
e_{k}\geq C_{1} k^{\frac{2p}{\tilde{\nu}(p-2)}}~~~~\forall k\geq 1,
\end{equation}
and
\begin{equation}
  b_{k}\geq C_{2}k^{\frac{2p}{\tilde{\nu}(p-2)}}~~~~\forall k\geq k_{0}\geq 1,
\end{equation}
where $C_{1},C_{2}$ are some positive constants and $k_{0}$ is a positive integer.

\end{proposition}
\begin{proof}
For $k\geq 1$ and any $0\leq \varepsilon<1$, we let $V_{\varepsilon}:=-(p-1)B|v_k|^{p-2}-\varepsilon$, where $v_{k}$ is the critical point of $I_{p}$ given in Proposition \ref{prop4-6}. Since $|v_{k}|\in H_{X,0}^{1}(\Omega)$, by Corollary \ref{corollary2-1} we can verify  $V_{\varepsilon}\in L^{\frac{p_{1}}{2}}(\Omega)$ holds for some $p_{1}>\tilde{\nu}$ and all $0\leq \varepsilon<1$. It derives from
Proposition \ref{prop2-5} that the Dirichlet eigenvalue problems of subelliptic Schr\"{o}dinger operators $-\triangle_{X}+V_{\varepsilon}$ are well-defined for all $0\leq \varepsilon<1$. The quadratic form of $-\triangle_{X}+V_{\varepsilon}$ is given by
\begin{equation}
\begin{aligned}\label{4-21}
  \mathcal{Q}_{\varepsilon}(u,v):&=\int_{\Omega}Xu\cdot Xv dx+\int_{\Omega}V_{\varepsilon}uv dx\\
  &=  \int_{\Omega}Xu\cdot Xv dx-\int_{\Omega}((p-1)B|v_k|^{p-2}+\varepsilon)uv dx.
\end{aligned}
\end{equation}
Then for any $0<\varepsilon<1$, we have
\begin{equation}
\begin{aligned}\label{4-22}
   &N(0,-\triangle_{X}-(p-1)B|v_k|^{p-2}-\varepsilon)\\
   &=\max\{\text{dim} V| V\subset H_{X,0}^{1}(\Omega)~\text{is a subspace such that~}~ \mathcal{Q}_{\varepsilon}(u,u)<0~~\forall u\in V\setminus\{0\}\}.
\end{aligned}
\end{equation}
In particular, \eqref{4-2} and \eqref{4-21} give $ \mathcal{Q}_{0}(u,u)=\langle I''_{p}(v_{k})u,u\rangle$ for all $u\in H_{X,0}^{1}(\Omega)$. Thus
\begin{equation}
\label{4-23}
m^*(v_k)=\max\{\text{dim} V| V\subset H_{X,0}^{1}(\Omega)~\text{is a subspace such that~}~\mathcal{Q}_{0}(u,u)\leq 0~~\forall u\in V\}.
\end{equation}
Combining \eqref{4-22}, \eqref{4-23}, Proposition \ref{prop4-6} and Proposition \ref{CLR}, we have for any $k\geq 1$,
\begin{equation}\label{4-24}
 C\int_{\Omega}\left((p-1)B|v_{k}|^{p-2}+\varepsilon\right)^{\frac{\tilde{\nu}}{2}}dx \geq N(0,-\triangle_{X}-(p-1)
 B|v_k|^{p-2}-\varepsilon)
  \geq m^*(v_k) \geq k,
\end{equation}
where $C>0$ is a positive constant. Taking $\varepsilon \to 0$ in \eqref{4-24}, we get
\begin{equation}\label{4-25}
\int_{\Omega}|v_k|^{\frac{(p-2)\tilde{\nu}}{2}} dx \geq \widetilde{C}\cdot k,
\end{equation}
where $\widetilde{C}>0$ is a positive constant. Besides, $DI_p (v_k)=0$ yields $B\int_{\Omega} |v_k|^pdx=\int_{\Omega} |Xv_k|^2dx$. It follows from \eqref{4-17} and \eqref{4-25} that
\begin{equation}\label{4-26}
\begin{aligned}
e_k\geq I_p(v_k)&=\frac{1}{2}\int_{\Omega}|Xv_k|^2dx-\frac{B}{p}\int_{\Omega}|v_k|^p dx\\
&=B\left(\frac{1}{2}-\frac{1}{p}\right)\int_{\Omega}|v_k|^p dx\\
&\geq CB\left(\frac{1}{2}-\frac{1}{p}\right)\left(\int_{\Omega}|v_k|^{\frac{(p-2)\tilde{\nu}}{2}} dx \right)^{\frac{2p}{\tilde{\nu}(p-2)}}\\
& \geq C_1 k^{\frac{2p}{\tilde{\nu}(p-2)}}~~~\forall~k\geq 1,
\end{aligned}
\end{equation}
where $C_{1}>0$ is a positive constant.

On the other hand, using Young's inequality and Corollary \ref{corollary2-1}, we obtain for any $u\in H_{X,0}^{1}(\Omega)$,
\begin{equation}
\begin{aligned}\label{4-27}
E_1(u)&=\frac{1}{2}\int_{\Omega}|Xu|^2dx-\int_{\Omega}F(x,u)dx-\psi(u)\int_{\Omega}G(x,u) dx\\
&\geq \frac{1}{2}\int_{\Omega}|Xu|^2dx-C\int_\Omega|u|^pdx-C\int_\Omega|u|dx-\left(\int_{\Omega}|\alpha(x)u(x)|dx+\frac{\beta}{\sigma+1}\int_{\Omega}|u|^{\sigma+1}dx\right)\\
      &\geq \frac{1}{2}\int_{\Omega}|Xu|^2dx-\widetilde{C_{1}}\int_{\Omega}|u|^{p}dx-\|\alpha\|_{L^{\frac{2\tilde{\nu}}{\tilde{\nu}+2}}(\Omega)}\|u\|_{L^{\frac{2\tilde{\nu}}{\tilde{\nu}-2}}(\Omega)}-\widetilde{C_{1}}\\
       &\geq \frac{1}{2}\int_{\Omega}|Xu|^2dx-\widetilde{C_{1}}\int_{\Omega}|u|^{p}dx-\widetilde{C_{2}}\|\alpha\|_{L^{\frac{2\tilde{\nu}}{\tilde{\nu}+2}}(\Omega)}\left(\int_{\Omega}|Xu|^2dx\right)^{\frac{1}{2}}-\widetilde{C_{1}}\\
       &\geq \frac{1}{4}\int_{\Omega}|Xu|^2dx-\widetilde{C_{1}}\int_{\Omega}|u|^{p}dx-\widehat{C_{1}},\\
\end{aligned}
\end{equation}
where $C,\widetilde{C_{1}},\widehat{C_{1}},\widetilde{C_{2}}$ are some positive constants. Now, substituting $B=2p\widetilde{C_{1}}$, \eqref{4-27} gives that
\[ E_1(u)\geq \frac{1}{2}I_{p}(u)-\widehat{C_{1}}~~~~\forall u\in H_{X,0}^{1}(\Omega), \]
which indicates
\[ b_k=\inf\limits_{h\in G_k}\max\limits_{u\in D_k}E_1(h(u))\geq \frac{1}{2}\inf\limits_{h\in G_k}\max\limits_{u\in D_k}I_{p}(h(u))-\widehat{C_{1}}=\frac{1}{2}e_{k}-\widehat{C_{1}}~~~~\forall k\geq 1. \]
Consequently, we can conclude from \eqref{4-26} that
\[  b_{k}\geq C_{2}k^{\frac{2p}{\tilde{\nu}(p-2)}}\]
holds for all $k\geq k_{0}$, where $C_{2}>0$ is a positive constant and $k_{0}\geq 1$ is a positive integer.
\end{proof}

Finally, we finish the proof of Theorem \ref{thm2}.
\begin{proof}[Proof of Theorem \ref{thm2}]
According to the proof of Theorem \ref{thm1},
it is sufficient to verify the validity of \eqref{3-57} under the assumption $(A2)$ in Theorem \ref{thm2}. Replacing Proposition \ref{prop3-4} by Proposition \ref{prop4-7} in the proof of Proposition \ref{prop3-5}, we conclude from \eqref{3-65} that
\begin{equation}
 b_{k_{2}+l} \leq b_{k_{2}}\exp\left(2A \sum_{k=k_{2}}^{k_{2}+l-1} b_{k}^{\frac{\sigma+1-\mu}{\mu}}\right).
\end{equation}
The assumption $(A2)$ gives
 \[ \frac{2p}{\tilde{\nu}(p-2)}>\frac{\mu}{\mu-\sigma-1},\]
i.e.,
\[ \frac{2p}{\tilde{\nu}(p-2)}\cdot\frac{\sigma+1-\mu}{\mu}<-1. \]
Hence, there exists a positive constant $\widetilde{M_{2}}>0$ such that
\begin{equation}\label{4-29}
  \sum_{k=k_{2}}^{\infty} b_{k}^{\frac{\sigma+1-\mu}{\mu}}\leq C_{2}^{\frac{\sigma+1-\mu}{\mu}}\sum_{k=k_{2}}^{\infty} k^{\frac{2p}{\tilde{\nu}(p-2)}\cdot\frac{\sigma+1-\mu}{\mu}} \leq \widetilde{M_{2}}<+\infty
\end{equation}
holds for sufficiently large $k_{2}$. Therefore, by \eqref{3-65} and \eqref{4-29}, we can also find a positive constant $\widetilde{M}$ such that $b_{k}\leq \widetilde{M}$ for all $k\geq 1$.
\end{proof}

\section{An example}
\label{Section5}
In this section, we present a simple example in which the admissible range of index $p$ given by the inequality $(A1)$ in Theorem \ref{thm1} will be wider than the range of $p$ given by the inequality $(A2)$ in Theorem \ref{thm2} (In fact, we can find a lot of such examples in degenerate cases).

\begin{example}
\label{ex5-1}
Let $X=(\partial_{x_{1}},x_{1}\partial_{x_{2}})$ be the Grushin type vector fields defined on  $\mathbb{R}^2$.  The Grushin operator  generated by $X$ is given by
\[ \triangle_{X}:=\frac{\partial^{2}}{\partial x_{1}^{2}}+x_{1}^{2}\frac{\partial^{2}}{\partial x_{2}^{2}}. \]
Assume $\Omega\subset\mathbb{R}^2$ is an open bounded smooth domain such that $0\in\Omega$. Clearly, $X$ satisfy the H\"{o}rmander's condition but fail to satisfy the M\'etivier condition. In addition, $\tilde{\nu}=3>n=2$. Denoting by $\lambda_{k}$  the $k$-th Dirichlet eigenvalue of $-\triangle_{X}$ on $\Omega$, it follows from \cite[Example 8.4]{Chen2021} that
\begin{equation}\label{6-1}
  \lambda_{k}\geq  C\left(\frac{k}{\ln k} \right)^{\frac{2}{\tilde{\nu}-1}}=C\left(\frac{k}{\ln k} \right)
\end{equation}
holds for sufficiently large $k$. That means $\vartheta=2$ and $\kappa=1$ in the condition (L) of Theorem \ref{thm1}.

In this case, we consider the problem \eqref{problem1-1} with $f(x,u)=u|u|^{p-2}$ such that $f(x,u)$ satisfies assumptions (H.1)-(H.3) with $p=\mu>2$. Also, we let $g(x,u)=|u|^{\sigma}$ with $\sigma=\frac{3}{4}p-1$ such that $0\leq \sigma<\mu-1=p-1$. It derives that $g(x,u)$ satisfies assumption (H.4) and
\[ (\sigma+1)\tilde{\nu}=\frac{9p}{4}>2p=\vartheta \mu,   \]
which implies the admissible range of index $p$ given in the inequality  $(A1)$ is wider than that in the inequality $(A2)$.
\end{example}


\Acknowledgements{Hua Chen is supported by National Natural Science Foundation of China (Grant No. 12131017) and National Key R\&D Program of China (no. 2022YFA1005602). Hong-Ge Chen is supported by National Natural Science Foundation of China (Grant No. 12201607), Knowledge Innovation Program of Wuhan-Shuguang Project (Grant No. 2023010201020286), China Postdoctoral Science Foundation (Grant No. 2023T160655) and Hubei Province Postdoctoral Funding Project. Jin-Ning Li is supported by China National Postdoctoral Program for Innovative Talents (Grant No.  BX20230270).}






\begin{thebibliography}{99}
\bibitem{Andrei2019}
Agrachev A, Barilari D, Boscain U. A comprehensive introduction to sub-Riemannian geometry. Cambridge: Cambridge University Press, 2020



\bibitem{Bahri1981}
{Bahri A. Topological results on a certain class of functionals and application. J Funct Anal, 1981, 41: 397-427}

\bibitem{Bahri-Berestycki1981}
{Bahri A, Berestycki H. A perturbation method in critical point theory and applications. Trans Amer Math Soc, 1981, 267: 1-32}


\bibitem{Bahri1988}
{ Bahri A, Lions P-L. Morse index of some min‐max critical points. I. Application to multiplicity results. Comm Pure Appl Math, 1988, 41: 1027-1037}



\bibitem{Bahri1992}
{Bahri A, Lions P-L. Solutions of superlinear elliptic equations and their Morse indices. Comm Pure Appl Math, 1992, 45: 1205-1215}

\bibitem{Bahouri2005}{
Bahouri H, Chemin J-Y, Xu C J. Trace and trace lifting theorems in weighted Sobolev spaces. J Inst Math Jussieu, 2005, 4: 509-552}


\bibitem{Bell1996}{
Bell D. Degenerate Stochastic Differential Equations and Hypoellipticity. London: Chapman \& Hall/CRC, 1995} 

\bibitem{MB2014}
{Bramanti M. An invitation to hypoelliptic operators and H\"{o}rmander's vector fields. Cham: Springer, 2014}

\bibitem{Capogna1993}
{Capogna L, Danielli D, Garofalo N. An embedding theorem and the Harnack inequality for nonlinear subelliptic equations. Comm Partial Differential Equations, 1993, 18: 1765-1794}

\bibitem{Capogna1994}
{ Capogna L, Danielli D, Garofalo N. The geometric Sobolev embedding for vector fields and the
isoperimetric inequality. Comm Anal Geom, 1994, 2: 203-215}

\bibitem{Capogna1996}
{Capogna L, Danielli D, Garofalo N. Capacitary estimates and the local behavior of solutions of nonlinear subelliptic equations. Amer J Math, 1996, 118: 1153-1196}

\bibitem{Chen2019}
{ Chen H, Chen H G. Estimates of eigenvalues for subelliptic operators on compact manifold. J Math Pures Appl, 2019, 131: 64-87}

\bibitem{Chen2021}
{Chen H, Chen H G. Estimates of Dirichlet eigenvalues for a class of sub‐elliptic operators. Proc Lond Math Soc, 2021, 122: 808-847}


\bibitem{Chen-Chen-Li2022}
{Chen H, Chen H G, Li J N. Asymptotic behaviour of Dirichlet eigenvalues for homogeneous
H\"{o}rmander operators and algebraic geometry approach. arXiv:2203.10450v2, 2022}

\bibitem{Chen-Chen-Yuan2022}
{Chen H, Chen H G, Yuan X R. Existence and multiplicity of solutions to semilinear Dirichlet
problem for subelliptic operator with a free perturbation. J Differential Equations, 2022, 341: 504-537}


\bibitem{Citti1995}
{Citti G. Semilinear Dirichlet problem involving critical exponent for the Kohn Laplacian. Ann Mat Pura Appl, 1995, 169: 375-392}

\bibitem{Citti1993}
{Citti G, Garofalo N, Lanconelli E. Harnack's inequality for sum of squares of vector fields plus a potential.  Amer J Math, 1993, 115: 699-734}

\bibitem{Cohn2007}
{Cohn W S, Lu G, Wang P. Subelliptic global high order Poincar\'{e} inequalities in stratified Lie
groups and applications. J Funct Anal, 2007, 249: 393-424}


\bibitem{Derridj1971}
{Derridj M. Un probleme aux limites pour une classe d'op{\'e}rateurs du second ordre hypoelliptiques. Ann. Inst. Fourier (Grenoble), 1971, 21: 99-148}

\bibitem{Derridj1972}
{Derridj M. Sur un th\'{e}or\`{e}me de traces. Ann Inst Fourier (Grenoble), 1972, 2: 73-83}


\bibitem{Dong1982}
{Dong G C, Li S. On the existence of infinitely many solutions of the Dirichlet problem for some nonlinear elliptic equations. Sci Sinica Ser A, 1982, 25: 468-475}

\bibitem{Driver2003}
{Driver B K. Analysis tools with applications. Lecture Notes, 2003}

\bibitem{Dugundji1951}
{Dugundji J. An extension of Tietze's theorem. Pacific J Math, 1951, 1: 353-367}

\bibitem{Frank2010}
{Frank R L, Lieb E H, Seiringer R. Equivalence of Sobolev inequalities and Lieb-Thirring inequalities. In: XVIth International Congress on Mathematical Physics, 2010, 523-535}


\bibitem{Garofalo1992}
{Garofalo N, Lanconelli E. Existence and nonexistence results for semilinear equations on the Heisenberg group. Indiana Univ Math J, 1992, 41: 71-98}

\bibitem{Garofalo1996}
{Garofalo N, Duy-Minh N. Isoperimetric and Sobolev inequalities for Carnot-Carathéodory spaces and the existence of minimal surfaces. Comm Pure Appl Math, 1996, 49: 1081-1144}

\bibitem{Hajlasz1998}
{Haj{\l}asz P, Koskela P. Isoperimetric inequalities and imbedding theorems in irregular domains. J London Math Soc, 1998, 58: 425-450}

\bibitem{Hajlasz2000}
{Hajłasz P, Koskela P. Sobolev met Poincar\'{e}. Mem Amer Math Soc, 2000, 145}

\bibitem{hormander1967}
{H\"{o}rmander L. Hypoelliptic second order differential equations. Acta Math, 1967, 119: 147-171}


\bibitem{Hormander1973}
{H\"{o}rmander L. An introduction to complex analysis in several variables. New York: Elsevier, 1973}


\bibitem{Jean2014}
{Jean F. Control of nonholonomic systems: From sub-Riemannian geometry to motion planning. Cham: Springer, 2014}


\bibitem{Jerison1986duke}
{Jerison D. The Poincar\'{e} inequality for vector fields satisfying H\"{o}rmander's condition. Duke Math J, 1986, 53: 503-523}


\bibitem{Jerison-lee1987}
{Jerison D, Lee J M. The Yamabe problem on CR manifolds. J Differential Geom, 1987, 25: 167-197}


\bibitem{Jerison1986}
{Jerison D, S\'{a}nchez-Calle A.
Estimates for the heat kernel for a sum of squares of vector fields. Indiana Univ Math J, 1986, 35: 835-854}

\bibitem{Jerison1987}
Jerison D, S\'{a}nchez-Calle A. Subelliptic, second order differential operators. In: Complex Analysis III, College Park, Md, USA 1985–86, vol. 1277. Berlin: Springer, 1987, 46-77

\bibitem{Jost1998}
{Jost J, Xu C J. Subelliptic harmonic maps. Trans Amer Math Soc, 1998, 350: 4633-4649}

\bibitem{Krasnosel'skii1964}
{Krasnosel'skii M A. Topological Methods in the Theory of Nonlinear Integral Equations. New York: A Pergamon Press Book The Macmillan Company, 1964}

\bibitem{Levin1997}
{Levin D, Solomyak M. The Rozenblum-Lieb-Cwikel inequality for Markov generators. J Anal Math, 1997, 71: 173-193}




\bibitem{Loiudice2007}
{Loiudice A. Semilinear subelliptic problems with critical growth on Carnot groups. Manuscripta Math, 2007, 124: 247-259}


\bibitem{Luyen2019}
{Luyen D T, Tri N M. On the existence of multiple solutions to boundary value problems for semilinear elliptic degenerate operators.
Complex Var Elliptic Equ, 2019, 64: 1050-1066
}


\bibitem{Maalaoui2013}
{Maalaoui A, Martino V. Multiplicity result for a nonhomogeneous Yamabe type equation involving the Kohn Laplacian. J Math Anal Appl, 2013, 399: 333-339}


\bibitem{Metivier1976}
{M{\'e}tivier G. Fonction spectrale et valeurs propres d'une classe d'op\'{e}rateurs non elliptiques. Comm Partial Differential Equations,
1976, 1: 467-519
}



\bibitem{Montgomery2002}
{Montgomery R. A tour of subriemannian geometries, their geodesics and applications. Providence, RI: American Mathematical Society, 2002}


\bibitem{Stein1985}
{Nagel A, Stein E M. Balls and metrics defined by vector fields I: Basic properties. Acta Math, 1985, 155: 103-147}


\bibitem{Rabinowitz1982}
{Rabinowitz P H. Multiple critical points of perturbed symmetric functionals. Trans Amer Math Soc, 1982, 272: 753-769}


\bibitem{Rabinowitz1986}
{Rabinowitz P H. Minimax methods in critical point theory with applications to differential equations. Providence, RI: American Mathematical Society, 1986}


\bibitem{Saloff-Coste1992}
{Saloff-Coste L. A note on Poincar\'{e}, Sobolev, and Harnack inequalities. Internat Math Res Notices, 1992, 1992: 27-38}


\bibitem{Sanchez-calle1984}
{S\'{a}nchez-Calle A. Fundamental solutions and geometry of the sum of squares of vector fields.
Invent Math, 1984, 78: 143-160}

\bibitem{Stein1976}
{Rothschild L P, Stein E M. Hypoelliptic differential operators and nilpotent groups.
Acta Math, 1976, 137: 247-320
}


\bibitem{Sqassina2006}
{Squassina M, Existence, multiplicity, perturbation, and concentration results for a class of quasi-linear elliptic problems. Electron J Differential Equations, 2006, Monograph 7}



\bibitem{Struwe1980}
{Struwe M. Infinitely many critical points for functionals which are not even and applications to superlinear boundary value problems. Manuscripta Math, 1980, 32: 335-364}



\bibitem{Struwe2000}
{Struwe M. Variational methods. New York: Springer, 2000}


\bibitem{Tanaka1989}
{Tanaka K. Morse indices at critical points related to the symmetric mountain pass theorem and applications. Comm Partial Differential Equations, 1989, 14: 99-128}


\bibitem{Murthy2008}
{Venkatesha Murthy M K. A class of subelliptic quasilinear equations. J Global Optim, 2008, 40: 245-260}


\bibitem{Willem1997}
{Willem M. Minimax theorems. Boston: Birkh\"{a}user, 1996 }


\bibitem{Xu1990}
{Xu C J. Subelliptic variational problems. Bull Soc Math France, 1990, 118: 147-169}


\bibitem{Xu1994}
{Xu C J. Semilinear subelliptic equations and Sobolev inequality for vector fields satisfying H\"{o}rmander's condition. Chinese J Contemp Math, 1994, 15: 183-192}


\bibitem{Xu1995}
{Xu C J. Existence of bounded solutions for quasilinear subelliptic Dirichlet problems.
J Partial Differential Equations, 1995, 8: 97-107}

\bibitem{Xu1996}
{Xu C J. Dirichlet problems for the quasilinear second order subelliptic equations. Acta Math Sin (Engl Ser), 1996, 12: 18-32 }

\bibitem{Xu1997}
{Xu C J, Zuily C. Higher interior regularity for quasilinear subelliptic systems.
Calc Var Partial Differential Equations, 1997, 5: 323-343
}

\bibitem{Yung2015}
{Yung P L. A sharp subelliptic Sobolev embedding theorem with weights. Bull Lond Math Soc, 2015, 47: 396-406}


\end{thebibliography}
\end{document}